\documentclass{amsart}[11pt]
 
\usepackage{amscd, amsfonts, amsmath,amsthm, amssymb, color,enumerate , epsf, faktor,graphicx, mathtools, 
 tikz, url ,verbatim ,wasysym , xypic}
 \usepackage{tikz-cd}

\usepackage[all]{xy}
\usetikzlibrary{matrix,arrows}
\usepackage[mathscr]{euscript}

\DeclarePairedDelimiter\floor{\lfloor}{\rfloor}

\usepackage[final]{hyperref}
\hypersetup{ colorlinks = true, linkcolor=blue,   citecolor = black}

\theoremstyle{plain}

 \newtheorem*{theorem*}{Theorem}
\newtheorem{thm}{Theorem}[subsection]
\newtheorem{prop}{Proposition}[subsection]
\newtheorem{defin}{Definition}[subsection]
\newtheorem{lemma}{Lemma}[subsection]
\newtheorem{coro}{Corollary}[subsection]
\newtheorem{remark}{Remark}[subsection]
\newcommand{\OK}{\mathcal{O}_K}

\newcommand{\OMa}{\mathcal{O}_{\mathcal{M}}}
\newcommand{\OLa}{\mathcal{O}_{\mathcal{L}}}
\newcommand{\ONa}{\mathcal{O}_{\mathcal{N}}}

\newcommand{\Z}{\mathbb{Z}}

\newcommand{\Qp}{\mathbb{Q}_p}
\newcommand{\Zp}{\mathbb{Z}_p}
\newcommand{\Tr}{\mathbb{T}_{\La/K}}
\newcommand{\OZ}{\mathcal{O}}

\newcommand{\Ka}{\mathcal{K}}
\newcommand{\Ma}{\mathcal{M}}
\newcommand{\La}{\mathcal{L}}
\newcommand{\Na}{\mathcal{N}}

\newcommand{\LTT}{L\{\{T_1\}\}\cdots \{\{T_{d-1}\}\}}

\newcommand{\T}{{\{\{T\}\}}}
\newcommand{\TrM}{\mathbb{T}_{\Ma/K}}

\newcommand{\TT}{\{\{T_1\}\}\cdots \{\{T_{d-1}\}\}}

\newcommand{\id}{(i_1,\dots,i_d)}

\newcommand{\Td}{T_1,\dots,T_{d-1} }
\newcommand{\Tdc}{T_1\cdots T_{d-1} }

\title{Explicit Reciprocity Laws for Higher Local Fields}

\author{Jorge Fl\'{o}rez}
\thanks{Support for this project was provided by a PSC-CUNY Award, jointly funded by The Professional Staff Congress and The City University of New York.}
\date{\today}
\address{Department of Mathematics, Borough of Manhattan Community College, City University of New York, 199 Chambers Street, New York, NY 10007, USA. jflorez@bmcc.cuny.edu}



\begin{document}
\maketitle


\begin{abstract}
 Using  previously constructed   reciprocity laws
for the generalized Kummer pairing of an arbitrary (one-dimensional) formal group, in this article a special consideration is given to  Lubin-Tate formal groups.  
In particular, this allows for a completely explicit description of the Kummer pairing  in terms of  multidimensional $p$-adic differentiation. The results obtained here constitute a generalization, to higher local fields, of the formulas of Artin-Hasse, Iwasawa, Kolyvagin  and Wiles.

\end{abstract}

\tableofcontents


\section{Introduction}

\subsection{Background}\label{Introduction}
For a prime $p>2$, let  $\zeta_{p^n}$ be a fixed $p^n$th root of unity and  $L$ be the cyclotomic field $\mathbb{Q}_p(\zeta_{p^n})$. The Hilbert symbol for $L$  
\begin{align*}
(,)_{p^n}:\ L^{\times}\times L^{\times}\to \langle \zeta_{p^n} \rangle  
\quad \mathrm{is\ defined\ as}\quad
(u\,,\,w)_{p^n}=\frac{\theta_L(u)\big(\sqrt[p^n]{w}\big)}{\sqrt[p^n]{w}},
\end{align*}
where $\theta_L:L^*\to \text{Gal}(L^{ab}/L)$ is Artin's local reciprocity map.  Iwasawa 
\cite{Iwasawa}  deduced explicit formulas for this pairing in terms of $p$-adic differentiation as follows
\begin{equation}\label{Iwasawa}
(u\,,\,w)_{p^n}=\zeta_{p_n}^{_{\,\text{Tr}_{L/\Qp}(\psi(w)\, \log u)\,/p^n}},\quad \mathrm{where} \quad \psi(w)=-\zeta_{p^n}\, 
w^{-1}\,\frac{dw}{d\pi_{_n}},
\end{equation}
and subject to the conditions $v_L(w-1)>0$ and $v_L(u-1)>2v_L(p)/(p-1)$. Here $v_L$ is the discrete valuation of $L$, $\pi_n$ is the uniformizer $\zeta_{p^n}-1$, and $dw/d\pi_n$ denotes $g'(\pi_n)$ for any power series $g(x)\in \Z_p[[X]]$ such that $w=g(\pi_n)$.

Iwasawa's formula stemmed from the Artin-Hasse formula
\begin{equation}\label{Hasselin}
(u\,,\,\zeta_{p^n})_{p^n}=\zeta_{p_n}^{_{\,\text{Tr}_{L/\Qp}(-\log u)\,/p^n}}, 
\end{equation}
where  $u$ is any principal unit in $L$.

 In this paper we will deduce generalized  formulas of (\ref{Iwasawa}) for the Kummer pairing  associated to a Lubin-Tate formal group  and an arbitrary higher local field of mixed characteristic (cf. Theorem \ref{Lubinazo}). Like Iwasawa, we will obtain these formulas  from a generalized Artin-Hasse formula for the Kummer pairing. The method we use to derive these formulas is inspired by that of Kolyvagin in \cite{koly}. This has  important consequences since Theorem \ref{Lubinazo} gives an exact generalization of Wiles \cite{Wiles} explicit reciprocity laws to arbitrary higher local fields.
  As a byproduct, we will obtain an exact generalization of \eqref{Iwasawa} (cf. Theorem \ref{Iwasawa-General}) and \eqref{Hasselin} (cf. Corollary \ref{Lemma-Ev-torison-3}) to higher local fields.
  
The main result in this paper, namely Theorem \ref{Lubinazo}, is supported by \cite{Florez} Theorem 5.3.1 which describes a general reciprocity law for the Kummer pairing associated to an arbitrary (one-dimensional) formal group  and an arbitrary higher local field in terms of $p$-adic multidimensional derivations. However it is not a straightforward consequence of it as the results in \cite{Florez} do not provide the sharpest results in the particular case when the Kummer pairing is associated to a Lubin-Tate formal group. In order to obtain the strongest possible results in Theorem \ref{Lubinazo} we need to make a considerable detour by relating \cite{Florez} Theorem 5.3.1  to   certain canonical multidimensional derivations associated to the torsion points of a Lubin-Tate formal group (cf. Section \ref{cann}) 
 and, also, by explicitly computing certain invariant attached to the Galois representation associated to the Tate module of the formal group (cf. Section \ref{Computations of the invariants}). 
  
   We point out that    Kurihara (cf. \cite{Kurihara} Theorem 4.4) and Zinoviev (cf. \cite{Zinoviev} Theorem 2.2 ) also have a generalization of \eqref{Iwasawa} for the \textit{generalized Hilbert symbol} associated to an arbitrary higher local field.  Theorem \ref{Iwasawa-General} further generalizes \eqref{Iwasawa}  to the Kummer pairing of an arbitrary Lubin-Tate formal group and an important family of higher local fields. In particular, when we take for the Lubin-Tate formal group   the multiplicative formal group $X+Y+XY$ and the higher local field to be $\Qp(\zeta_{p^n})\TT$, Theorem \ref{Iwasawa-General} coincides with the results of  Kurihara and Zoniviev.  
   
   In \cite{Fukaya}, Fukaya  remarkably describes also  similar formulas to those of Theorem \ref{Iwasawa-General} for  the Kummer pairing  associated to an arbitrary $p$-divisible group $G$. Even though Fukaya's formulas  encompass also arbitrary higher local fields (containing the $p^n$th torsion group of $G$), Theorem \ref{Iwasawa-General} in its specific conditions is sharper than the results in \cite{Fukaya}  for Lubin-Tate formal groups as we explain in more detail below.

Finally, it is interesting to note that in Zinoviev's work there is  also stated a higher dimensional version of \eqref{Hasselin} for the generalized Hilbert symbol (cf. \cite{Zinoviev} Corollary 2.1). Our Corollary  \ref{Lemma-Ev-torison-3} further extends   \eqref{Hasselin} to
arbitrary  Lubin-Tate formal groups and arbitrary higher local fields, in particular subsuming the formulas of Zinoviev. Moreover, we will prove  stronger results (cf. Proposition \ref{Lemma-Ev-torison-2} and Equation \eqref{Stronger-Artin-Hasse}) which are not found, \textit{a priori}, in any of the formulas in literature.

\subsection{Description of the results}\label{Description of the results}

Let $K/\Qp$ be  a local field with ring of integers $C=\OK$,
 $k_K$ its residue field and $q=|k_K|$. Denote by $\varrho$ the ramification index $e(K/\Qp)$. Fix a uniformizer $\pi$ for $K$.
Let $\Lambda_\pi$ be the subset of $C[[X]]$ consisting of the series $f$ such that
\begin{enumerate}
\item $f(X)\equiv \pi X \pmod{ \deg 2}$.
\item $f(X)\equiv  X^q \pmod{ \pi}$.
\end{enumerate}
For a fixed $f\in \Lambda_{\pi}$, let $F=F_f$  be the Lubin-Tate formal group associated to $f$ (cf. \cite{cassels} VI \S 3.3 or \cite{koly} \S7). 
The ring $C$ is identified with  the endomorphism ring of $F$ as follows: 
$C\to \text{End}(F):a\to [a]_f(X):=aX+\cdots$. Via this identification, we denote by $f^{(n)}(X)$  the power series $[\pi^n]_f(X)$, for $n\geq 1$.
Let $\kappa_{f,n}\,(\simeq C/\pi^n C)$ be the $\pi^n$th torsion group of $F$, $\kappa_{f}=\varprojlim \kappa_{f,n}\,(\simeq C)$ be the Tate module and 
$\kappa_{f,\infty}=\cup_{n\geq 1} \kappa_{f,n}$ ($\simeq K/C$).  
We will fix  a generator $e_{f}$ for $\kappa_f$ and let $e_{f,n}$ be the corresponding reduction to  the group $\kappa_{f,n}$. Denote by $\Ka$  the 
$d$-dimensional standard local field $K\TT$, and by  
 $K_{\pi,n}$, $K_{\pi,\infty}$, $\Ka_{\pi,n}$ and $\Ka_{\pi,\infty}$, respectively, the  fields 
 $K(\kappa_{f,n})$, $K(\kappa_{f,\infty})$, $\Ka(\kappa_{f,n})$ and $\Ka(\kappa_{f,\infty})$, respectively. 
Since we already fixed $\pi$ and $f$, for convenience sometimes we will omit the subscripts $\pi$ and $f$ in the above notation. 

Finally, we will denote by $l_f$ the principal logarithm of $F$ (cf. \cite{koly} Section 1.3), i.e., $l_f(X)$ is the power series $\sum_{i=1}^{\infty} (a_i/i)X^i$ with $a_i\in C$ and $a_1=1$, such that $l_f(F(X,Y))=l_f(X)+l_f(Y)$. Moreover, $l_f\circ [a]_f=al_f$ for all $a\in C$.

 In order to describe our formulas, let $\La\supset K$ be a $d$-dimensional local field containing the torsion group
$\kappa_{f,n}$, with ring of integers $\OLa$ and  maximal ideal $\mu_\La$. Fix a system of local uniformizers $T_1,\dots, T_{d-1}$ and $\pi_\La$ of $\La$.  
We will denote by $F(\mu_{\La})$ the set $\mu_{\La}$ endowed with
the group structure of $F$. For $m\geq 1$ we let  $\La_m=\La(\kappa_{f,m})$ and 
we fix a uniformizer $\gamma_m$ for $\La_m$.  
 
We define the Kummer pairing (cf. \cite{Florez} \S2.2)
\[
  (,)_{\La,n}: K_d(\La) \times F(\mu_{\La}) \to \kappa_{f,n} \quad \mathrm{by}\quad
           (\alpha, x)\mapsto (\alpha, x)_{\La,n}:=\Upsilon_\La(\alpha)(z)\ominus_f z,
\]
where $K_d(\La)$ is the $d$th 
Milnor $K$-group of $\La$ (cf. \cite{Florez}\S2.1.1), 
$\Upsilon_\La:K_d(\La)\to G_\La^{ab}$  is  
Kato's reciprocity map for $\La$ (cf. \cite{Florez} \S2.1.4), $f^{(n)}
(z)=x$ and $\ominus_f$ is the subtraction 
in the formal group $F_f$. 

To be consistent  with \cite{Florez}, sometimes we will use the following notation. Since $(\alpha,x)_{\La,n}\in \kappa_{f,n}$ and we are assuming $e_n=e_{f,n}$
is a generator  of $\kappa_{f,n}$ as a $C/\pi^nC$-module, then there exists 
an element $(\alpha,x)_{\La,n}^1\in C/\pi^nC$ such that
\begin{equation}\label{notation-kummer-pairing}
(\alpha,x)_{\La,n}=\big[\,(\alpha,x)_{\La,n}^1\,\big]_f(e_{f,n}).
\end{equation}

The main result in this paper is 
the following (cf. Theorem \ref{Lubinazo} for the precise formulation).

\begin{theorem*}
Let $r$ be maximal and $r'$  minimal  such that $K_{\pi,r}\subset \La\cap K_{\pi,\infty}\subset K_{\pi,r'}.$ Take $s\geq \max\{r', n+r+\log_q(e(\La/\Ka_{\pi,r}))\}$; where $e(\La/\Ka_{\pi,r})$ is the ramification index of the extension $\La/\Ka_{\pi,r}$. 
Then 
\begin{equation}\label{maine}
\big(\,N_{\La_s/\La}(\alpha)\,,x\,\big)_{\La,n}
=\big[\, 
    \mathbb{T}_{\La_s/K}
      \big(\, 
         QL_s(\alpha)\, l_f(x)\,
       \big)\,
           \big]_f (e_{f,n}),
\end{equation}
where
\[
QL_s(\alpha)=\frac{  T_1\cdots T_{d-1} }{\, \pi^{s}\,l_f'(e_{f,s})\,\frac{\partial e_{f,s}} {\partial T_d} }\, \frac{ \det \left[ \frac{\partial a_i} {\partial T_j}\right]_{1\leq i,j\leq d}}{a_1\cdots a_d},
\]
for all $x\in F(\mu_\La)$ and  all  $\alpha=\{a_1,\dots, a_d\}$ in $ \bigcap_{t\geq s} N_{\La_t/\La_s} \big(\,K_d(\La_t)\,\big)$. Here $N_{\La_s/\La}$ and $N_{\La_t/\La_s}$ denote the norm on Milnor $K$-groups (cf. \cite{Florez} \S2.1.2), $l_f$ is the logarithm of the formal group, $\mathbb{T}_{\La_s/K}$ denotes the generalized trace and $\frac{\partial a_i} {\partial T_j}$ denotes the partial derivative of an element in $\La_s$ with respect to the system of local uniformizers $T_1,\dots, T_{d-1}, T_d=\gamma_s$ of $\La_s$ (cf. Section  \ref{Terminology and Notation}).
\end{theorem*}

Just like in the work of Iwasawa, the formula (\ref{maine}) will be deduced
from the following generalized Artin-Hasse formula (cf. Proposition \ref{Artin-Hasse})
 \begin{equation}\label{Artinios}
 \big(\, \{T_1,\dots, T_{d-1},u\},\, e_{f,t}\, \big)_{\Ma,k}
    =\left[\mathbb{T}_{\Ma/S}\left(\, \log(u)\, \left( -\frac{1}{\pi^t}\right) \, \right) \right]_f(e_{f,k}),
 \end{equation}
for  $k>0$ and $t$ sufficiently large with respect to $k$ ( more concretely, in the notation of Section \ref{Computations of the invariants},  $(k,t)$ is an \textit{admissible pair}), and where $\Ma=\La_t$, $ u\in V_{\Ma,1}=\{u\in \OMa:v_{\Ma}(u-1)>v_{\Ma}(p)/(p-1)\}$, $\log$ is the usual logarithm and $v_{\Ma}$ is the discrete valuation of $\Ma$.

The above theorem is an exact generalization of Wiles reciprocity laws to arbitrary higher local fields (cf. \cite{Wiles} Theorem 1.). 

By doing a detailed analysis of how the formulas \eqref{maine} are  transformed  when  varying the uniformizer $\pi$ of $K$  and the power series $f\in \Lambda_{\pi}$   we may prove, in Theorem \ref{Iwasawa-General}, the following higher dimensional version of Iwasawa's reciprocity laws \eqref{Iwasawa} for $\La=\Ka_{\pi,n}$:
\begin{equation}\label{Iwasawa-Gen-Intro}
(\,\alpha\,,x\,)_{\La,n}
=\big[\, 
    \mathbb{T}_{\La/K}
      \big(\, 
         QL_n(\alpha)\, l_f(x)\,
       \big)\,
           \big]_f (e_{f,n})
\end{equation}
for all $\alpha\in K_d(\La)$ and all $x\in F_f(\mu_{\La})$ such that $v_{\La}(x)\geq 2\,v_{\La}(p)/(\varrho\,(q-1))$. In particular, taking $K=\Qp$, $f(X)=(1+X)^{p^n}-1$, $F_f(X,Y)$ the multiplicative formal group $X+Y+XY$ and $\La$ the  cyclotomic higher local field $\Qp(\zeta_{p^n})\TT$, then \eqref{Iwasawa-Gen-Intro} coincides with   \cite{Kurihara} Theorem 4.4 and  \cite{Zinoviev} Theorem 2.2.

We must remark also that Fukaya \cite{Fukaya} has similar formulas to \eqref{Iwasawa-Gen-Intro} that, moreover, extend to arbitrary formal groups and arbitrary higher local fields. However, for Lubin-Tate formal groups the formula \eqref{Iwasawa-Gen-Intro} is sharper for $\La=\Ka_{\pi,n}$ , as the condition on $x\in F(\mu_{\La})$ in \cite{Fukaya} is $v_{\La}(x)>2v_{\La}(p)/(p-1)+1$.

In the deduction of \eqref{Iwasawa-Gen-Intro} we will  also prove, in Corollary \ref{Lemma-Ev-torison-3}, the following Artin-Hasse formula:   Let $\Ma$ be an arbitrary $d$-dimensional local field with $T_1,\dots, T_{d-1},\pi_{\Ma}$ as a system of local uniformizers  such that $\Ma \supset \Ka_{\pi,n}$, then
\begin{equation}\label{Artin-Hasse-Intro}
\,
\big(\,\{T_1,\dots, T_{d-1},e_{g,n}\}\,,\,x\,\big)_{\Ma,n}
=\left[ 
\mathbb{T}_{\Ma/S}\left(\,\frac{1}{\,\xi^n\,l'_g(e_{g,n})\,e_{g,n}\,}\, l_f(x)\,\right)\,
\right]_f \,(e_{f,n})
\end{equation}
for all $x\in F_f(\mu_{\Ma})$, where $g$ is a Lubin-Tate series in $\Lambda_{\pi}$ which is also a monic polynomial and $e_{g,n}=[1]_{f,g}(e_{f,n})$; here $[1]_{f,g}$ is the isomorphism of $F_f$ and $F_g$ congruent to $X\pmod{\deg 2}$ and $l_g$ is the logarithm of $F_g$. By  taking $K=\Qp$, $f(X)=(1+X)^{p^n}-1$, $F_f(X,Y)$ the multiplicative formal group $F_m(X,Y)=X+Y+XY$ and $\Ma$ the  cyclotomic higher local field $\Qp(\zeta_{p^n})\TT$,  then \eqref{Artin-Hasse-Intro} coincides with the Artin-Hasse formula of   Zinoviev in \cite{Zinoviev} Corollary 2.1 (25).

Furthermore, \eqref{Artin-Hasse-Intro} will be deduced as a consequence of the following stronger result (cf. Proposition \ref{Lemma-Ev-torison-2}). 
Let $\La =\Ka_{\pi,n}$ and take $e_{g,n}$ and $l_g$ as above, then
\begin{equation}\label{Artin-Hasse-complete}
\,
\big(\,\{u_1,\dots, u_{d-1},e_{g,n}\}\,,\,x\,\big)_{\La,n}
=\left[ 
\Tr\left(\,\frac{\det \left[\frac{\partial u_i}{\partial T_j}\right]}{u_1\cdots u_{d-1}}\,\frac{\Tdc}{\,\xi^n\,l'_g(e_{g,n})\,e_{g,n}\,}\, l_f(x)\,\right)\,
\right]_f \,(e_{f,n})
\end{equation}
for all units $u_1,\dots, u_{d-1}$ of  $\La$ and  all $x\in F_f(\mu_{\La})$.

 Moreover, in the particular situation where  $K=\Qp$, $\pi=p$, $f(X)=(X+1)^p-1$, $F_f(X,Y)=F_m(X,Y)$, $l_f(X)=\log(X+1)$, $\La=\Qp(\zeta_{p^n})\TT$,  we further have  an additional formula ( cf. Equation \eqref{Stronger-Artin-Hasse}):
\begin{equation}\label{Hasse-sto}
\big(\,\{u_1,\dots, u_{d-1},\zeta_{p^n}\}\,,\,x\,\big)_{\La,n}
=\left[ 
\mathbb{T}_{\La/\Qp}\left(\,\frac{1}{\,p^n\,}\, \frac{\det \left[\frac{\partial u_i}{\partial T_j}\right]}{u_1\cdots u_{d-1}} \,l_{f}(x)\,\right)\,
\right]_{f} (e_{f,n})
\end{equation}
for all units $u_1,\dots, u_{d-1}$ of  $\La$ and all $x\in F_f(\mu_{\La})$.  For $u_1=T_1,\dots, u_{d-1}= T_{d-1}$ we obtain  \cite{Zinoviev} Corollary 2.1 (24).

The sharper Artin-Hasse formulas \eqref{Artin-Hasse-complete} and \eqref{Hasse-sto} are not contained in any of the reciprocity laws in the literature.

This paper is organized as follows. In Section \ref{Canonical derivations and logarithmic derivatives} we formally construct  the canonical derivations and logarithmic derivatives describing the Kummer pairing. In Section \ref{Computations of the invariants} we derived an Artin-Hasse type formula (\ref{Artinios}) for the generalized Kummer pairing. 
In Section \ref{Generalized Kolyvagin formulas} we derived the formula (\ref{maine}) from the Artin-Hasse formula (\ref{Artinios}).  Finally, in Section \ref{Comparison of formulas} we will prove the formulas \eqref{Iwasawa-Gen-Intro}, \eqref{Artin-Hasse-Intro}, \eqref{Artin-Hasse-complete} and \eqref{Hasse-sto}.

\subsection{Notation}\label{Terminology and Notation}
We will fix a prime number $p>2$. 
If $x$ is a real number then $\floor*{x}$ 
denotes the greatest integer $\leq x$.

 We will denote by $D(M/L)$ the different of a finite  extension of local fields $M/L$.

 If $R$ is a discrete valuation ring, the symbols  $v_R$, $\mathcal{O}_R$, $\mu_R$ and $\pi_R$
will always denote the valuation, ring of integers, maximal ideal, and some fix uniformizer of $R$, respectively.

 Let $L$ be a complete discrete valuation field and define
\[
   \La=L\{\{T\}\}=\left\{\; \sum_{-\infty}^{\infty}a_iT^i:\ a_i\in
   L,\ \inf v_L(a_i)>-\infty,\ \lim_{i\to -\infty}v_L(a_i)
    =+\infty \  \right\}.
\]
Let $v_{\La}(\sum a_iT^i)=\min_{i\in \Z} v_L(a_i)$, so $\OZ_{\La}=\OZ_{L}\{\{T\}\}$ and
$\mu_{\La}=\mu_L\{\{T\}\}$. Observe that the residue field $k_{\La}$ of $\La$ is $k_L((\overline{T}))$, where $k_L$ is the residue field of  $L$. Associated to $\La$ we have the map $c_{\La/L}:\La\to L$ such that $c_{\La/L}(\sum_{-\infty}^{\infty}a_iT^i)=a_0$.

We define $\La=\LTT$  and $c_{\La/L}$ inductively.  We also define $\frac{\partial \alpha}{\partial T_k}:\OLa\to \OLa$, $k=1,\dots, d-1$ to be the partial derivative of $\alpha \in \OLa$ with respect to its canonical Laurent expansion in $\La$ (cf. \cite{Florez} Definition 4.1.2).

 For a $d$-dimensional local field $\La\supset K$ let $T_1,\dots,T_{d-1}$ and $\pi_{\La}$ denote a system of uniformizers,  and let $k_\La=\mathbb{F}((\overline{T}_1))\cdots ((\overline{T}_{d-1}))$ be its residue field. Let $\La_{(0)}$ be the standard field $L_{(0)}\TT$, where the one-dimensional local field $L_{(0)}$ is  the maximal unramified extension of $K$ contained in $\La$. In particular, $\La/\La_{(0)}$ is a finite totally ramified extension. 

For $g(X)= a_0+\cdots+a_kX^k\in \mathcal{O}_{\La_{(0)}}[X]$ we denote by $\frac{\partial g}{\partial T_i}(X)$, $i=1,\dots, d-1$, the polynomial
\[
\frac{\partial a_k}{\partial T_i}X^k+\cdots+\frac{\partial a_0}{\partial T_i}\in \mathcal{O}_{\La_{(0)}}[X] \quad (i=1,\dots, d-1)
\] 
and 
\[\frac{\partial g}{\partial T_d}(X)=g'(X)\quad (i=d).
\]
If $a\in \OLa$, let $g(X)\in \mathcal{O}_{\La_{(0)}}[X]$ such that $a=g(\pi_\La)$. Then we denote by $\frac{\partial a}{\partial T_i}$ the element $\frac{\partial g}{\partial T_i}(\pi_\La)$, $i=1,\dots, d$.

 We define the generalized trace $\mathbb{T}_{\La/K}$ to be the composition $\text{Tr}_{L_{(0)}/K}\circ c_{\La_{
 (0)}/L_{(0)}} \circ \text{Tr}_{\La/\La_{(0)}}$. We let $D(\La/K)$ denote the different with respect to the pairing $\langle,\rangle:\La\times \La\to K$: $(x,y)\to \mathbb{T}_{\La/K}(xy)$.
 
 Let $\Omega_{\OK}(\OLa)$ be the module of K\"{a}hler differentials of $\OLa$ over $\OK$, and  $\mathfrak{D}(\La/K)$ be the annihilator ideal of the torsion part of $\Omega_{\OK}(\OLa)$. We denote by $\hat{\Omega}_{\OK}(\OLa)$ its $p$-adic completion and by $\hat{\Omega}^{\,d}_{\OK}(\OLa)$ the $\OLa$-module $\wedge_{\OLa}^d \hat{\Omega}_{\OK}(\OLa)$. From \cite{Florez} Section 6.5, we have that $\hat{\Omega}^{\,d}_{\OK}(\OLa)$ is generated by 
  $dT_1\wedge \cdots \wedge dT_{d-1}\wedge d\pi_{\La}$ and, moreover, there is an isomorphism of $\OLa$-modules 
 \begin{equation}\label{Khal}
 \hat{\Omega}^{\,d}_{\OK}(\OLa)\
\cong \
 \OLa/\mathfrak{D}(\La/K).
 \end{equation}
 We also have that $\mathfrak{D}(\La/K)\,|\,D(\La/K)$. In particular, if $\La$ is a standard $d$-dimensional local field, then $\mathfrak{D}(\La/K)=D(\La/K)$.
\section{Canonical derivations and logarithmic derivatives}\label{Canonical derivations and logarithmic derivatives}

In this section we will give an explicit construction of the logarithmic derivative $QL_s$
in the formula (\ref{maine}). Then, in Section \ref{Generalized Kolyvagin formulas} we will show that (\ref{maine}) holds.

The following two propositions of Kolyvagin \cite{koly} will 
be needed in the construction 
of the logarithmic derivative and in the deduction of its main properties; we state them here for easy reference. Let $\overline{K}$ be a fixed algebraic closure of $K$ and   $\mathcal{O}_{\overline{K}}$ be its ring of integers.
Denote by $\Omega_{\OK}(\mathcal{O}_{\,\overline{K}})$  the module of differentials of $\mathcal{O}_{\,\overline{K}}$ with respect to $\OK$.

\begin{prop}\label{generators}
       Let $w_n=l_f'(e_n)\,de_n\in \Omega_{\OK}(\mathcal{O}_{\,\overline{K}})$, $n\geq 1$. Then $\{w_n\}_{n\geq1}$ generate 
       $\Omega_{\OK}(\mathcal{O}_{\,\overline{K}})$ as an
       $\mathcal{O}_{\,\overline{K}}$-module, and the following compatibility relationship holds
       \begin{equation}\label{relationdifferential}
        w_{n}=\pi w_{n+1}.
       \end{equation}
       
\end{prop}
\begin{proof}
cf. \cite{koly} Proposition 7.9.
\end{proof}

Let  $K_n=K(\kappa_n)$ and $K_{\infty}=K(\kappa_{\infty})$, 
where $\kappa_{\infty}=\cup \kappa_n$. Fix a uniformizer $\pi_n$ of $K_n$ and let 
\[
\Ka_n=K_n\TT\quad  \mathrm{and}\quad \Ka=K\TT.
\] 
The extension $K_n/K$ 
is totally ramified and $e_n$ is a uniformizer for $K_n$. Therefore, we will assume from now on that $\pi_n=e_n$.
Moreover, $[K_n:K]=q^n-q^{n-1}$ and $|\kappa_{n}|=q^n$, where $q$ is the size of the residue field of $K$, i.e, $|k_K|=q$ (cf. \cite{koly} Section 7.1.2).

Let   $\tau:\text{Gal}(\overline{K}/K)\,\to\, \OK^*$ be the Galois  representation 
induced by the action of $\text{Gal}(\overline{K}/K)$ on the Tate module $\kappa=\varprojlim \kappa_n\,(\simeq C)$  (cf. $\S$ 5.2,  
Equation (42) of \cite{Florez} ). This induces an embedding 
 \begin{equation}\label{Lubin-Tate-isom}
\tau_n: \text{Gal}(\Ka_n/\Ka)\to (\OK/\pi^n\OK)^*,
 \end{equation}
 which turns out to be an isomorphism since $[\Ka_n:\Ka]$  and $|(\OK/\pi^n\OK)^*|$ are both equal to $q^n-q^{n-1}$.

Finally, we notice that   $\text{Gal}(\overline{K}/K)$ acts on $\Omega_{\OK}(\mathcal{O}_{\overline{K}})$  as follows:  $(b\,da)^g:=b^gda^g$, for $a,b\in \mathcal{O}_{\overline{K}}$ and $g\in \text{Gal}(\overline{K}/K)$. The following proposition describes the behavior of the  differentials $w_n$  under this action.

\begin{prop}\label{galois-action}
       $w_n^g=\tau(g)w_n$ for all $g\in\normalfont  \text{Gal}(\overline{K}/K)$. 
\end{prop}
\begin{proof}
cf. \cite{koly} $\S$ 7.2.3.
\end{proof}

For $d$-dimensional local fields the Galois action on the module of K\"{a}hler differentials is defined in a similar fashion. More specifically, if  $\La$ is a finite extension of $\Ka$ we define the action of $\text{Gal}(\overline{\Ka}/\Ka)$ on $\hat{\Omega}_{\OK}(\OLa)$,
where $\overline{\Ka}$ is a fixed algebraic closure of $\Ka$, as follows:  
$(b\,da)^g:=b^gda^g$, for $a,b\in \OLa$ and $g\in \text{Gal}(\overline{\Ka}/\Ka)$. 
Moreover, we extend this action to  $\hat{\Omega}_{\OK}^d(\OLa)$ as
\begin{equation}\label{action-diff}
(b\,da_1\wedge \cdots \wedge da_d)^g:=b^g(da_1)^g\wedge \cdots \wedge (da_d)^g=b^gda_1^g\wedge \cdots \wedge da_d^g,
\end{equation}
for all $g\in \text{Gal}(\overline{\Ka}/\Ka)$ and  all $b,a_1,\dots, a_d \in \OLa$.
 
\subsection{The canonical derivations $Q_{\Ma,s}$}\label{cann} Let $\Ma$ be a finite extension of $\Ka_s$ with   $T_1,\dots,T_{d-1},\pi_\Ma$ as a system of local uniformizers. Let $\pi_1$ be a uniformizer for $K_1=K(\kappa_1)$.
We define
\[
P_\Ma:=\big(\,1/(\,\pi_1D(\Ma/K)\,)\,\big)\OMa=\{x\in \Ma:\ v_{\Ma}(x)\geq -v_{\Ma}(\pi_1D(\Ma/K)) \}.
\] 
In this section we are going to introduce a 
$d$-dimensional derivation  
\begin{equation*}\label{Lubinderi}
Q_{\Ma,s}:\mathcal{O}_\Ma^d\to P_\Ma/(\pi^s/\pi_1)P_\Ma,
\end{equation*}
over $\mathcal{O}_K$  which is fundamental in the construction of the reciprocity laws (cf. \cite{Florez} Section 4.2 for the definition and properties of $d$-dimensional derivations). Before we can define $Q_{\Ma,s}$, we need the following technical result.

\begin{lemma}\label{tech-lemma} Let  $b'\in \mathcal{O}_\Ma$ be 
such that 
$$dT_1\wedge \cdots \wedge dT_{d-1}\wedge w_s
=b'\, dT_1\wedge \cdots \wedge dT_{d-1}\wedge d\pi_\Ma,$$ 
as elements in $\hat{\Omega}_{\mathcal{O}_{K}}^d(\OMa)$; such element exists according to  Section \ref{Terminology and Notation}. Then $$b'\OMa=D(\Ma/K_s).$$
 Moreover, if  $b^*\in \OMa$ is another element 
such that $dT_1\wedge \cdots \wedge dT_{d-1}\wedge w_s
=b^*\, dT_1\wedge \cdots \wedge dT_{d-1}\wedge d\pi_\Ma$, then
$$b'\equiv b^* \pmod{D(\Ma/K)}.$$
\end{lemma}
\begin{proof}
 We first need to  introduce the auxiliary fields
 \[
 \tilde{\Ma}_{(0)}:=\tilde{M}_{(0)}\TT 
 \quad \mathrm{ and }\quad  
 \Ma_{(0)}:=M_{(0)}\TT, 
\] 
where the local field $\tilde{M}_{(0)}$ is the maximal unramified extension of $K_s$  contained in $\Ma$, and  
  the local field $M_{(0)}$  is the maximal unramified extension of $K$ contained in $\Ma$. 
 We denote by $\mathcal{O}_{\tilde{\Ma}_{(0)}}$ and $\mathcal{O}_{\Ma_{(0)}}$
  the ring of integers of $\tilde{\Ma}_{(0)}$ and $\Ma_{(0)}$, respectively. 
  We have that  $\Ma/\tilde{\Ma}_{(0)}$ and $\Ma/\Ma_{(0)}$ 
   are totally ramified extensions,
    and 
   $\tilde{\Ma}_{(0)}/\Ka_s$ and $\Ma_{(0)}/\Ka$ 
   are unramified. 
   
Let    $n=[\Ma:\Ma_{(0)}]$.   From the total ramification of $\Ma/\Ma_{(0)}$ we have that $\OMa=\mathcal{O}_{\Ma_{(0)}}[\pi_\Ma]$.  In particular $\pi_\Ma,\dots, \pi_\Ma^n$ is a basis for $\mu_{\Ma}$, the maximal ideal of $\OMa$, over $\mathcal{O}_{\Ma_{(0)}}$. Thus $e_s=b_1\pi_\Ma+\cdots +b_n\pi_\Ma^n$ for some  polynomial $g(X)=b_1X+\cdots+b_nX^n\in\mathcal{O}_{\Ma_{(0)}}[X] $; recall from  above that $e_s$ is a uniformizer of $K_s$, and consequently of $\Ka_s$, so $e_s\in \mu_{\Ma}$. 
  Since $\tilde{\Ma}_{(0)}\supset \Ma_{(0)}\Ka_s\supset \Ma_{(0)}$, the polynomial $g(X)-e_s$ belongs to $\mathcal{O}_{\tilde{\Ma}_{(0)}}[X]$ as well  and therefore it is divisible by $p(X)$, the minimal polynomial of $\pi_\Ma$ over $\tilde{\Ma}_{(0)}$, i.e., 
  \[
  g(X)-e_s=q(X)p(X),
  \]
   for some $q(X)\in \mathcal{O}_{\tilde{\Ma}_{(0)}}[X]$. Since $\Ma/\tilde{\Ma}_{(0)}$ is totally ramified, the polynomial $p(X)$ is an Eisenstein polynomial and so its constant coefficient is a uniformizer of $\tilde{\Ma}_{(0)}$ (consequently of $\Ka_s$). This implies that the constant coefficient of $q(X)$ is a unit in  $\mathcal{O}_{\tilde{\Ma}_{(0)}}$, and hence $q(\pi_{\Ma})$ is a unit in $\OMa$. 
   
   Differentiating  now $  g(X)-e_s=q(X)p(X)$ with respect of $X$, and evaluating at $X=\pi_\Ma$, we obtain
  \[
g'(\pi_\Ma)=q(\pi_\Ma)p'(\pi_\Ma).  
  \] 

On the other hand, for the finite extension $\Ma/\Ka_s$ we have that  the different $D(\Ma/\Ka_s)$ coincides with $D(\Ma/\tilde{\Ma}_{(0)})$ (since $D(\tilde{\Ma}_{(0)}/\Ka_s)=\OMa$ as $\tilde{\Ma}_{(0)}/\Ka_s$ is unramified), which in turn is equal to $p'(\pi_{\Ma})\OMa$ ( since the the extension $\Ma/\tilde{\Ma}_{(0)}$ is totally ramified and so  $\OMa=\mathcal{O}_{\tilde{\Ma}_{(0)}}[\pi_\Ma]\,$).
Therefore  since $q(\pi_\Ma)$ is a unit we have  $g'(\pi_\Ma)\OMa=p'(\pi_{\Ma})\OMa=D(\Ma/\tilde{\Ma}_{(0)})=D(\Ma/\Ka_s)=D(\Ma/K_s)$. Now, from $g(\pi_\Ma)=e_s$ and the properties of K\"{a}hler  differentials we obtain the identity
\[
de_s=\frac{\partial g}{\partial T_1}(\pi_\Ma)dT_1
+\cdots +
\frac{\partial g}{\partial T_{d-1}}(\pi_\Ma)dT_{d-1}+g'(\pi_{\Ma})d\pi_{\Ma}
\]
in $\hat{\Omega}_{\OK}(\OMa)$; recalling that $g(X)\in\mathcal{O}_{\Ma_{(0)}}[X] $.
By the alternating properties of the wedge product we have 
\begin{equation}\label{new-diff-iden}
dT_1\wedge \cdots \wedge de_s=g'(\pi_\Ma)\,dT_1\wedge\cdots \wedge d\pi_\Ma.
\end{equation}
Therefore letting 
\begin{equation}\label{myst-b}
b'=l_f'(e_s)g'(\pi_\Ma)
\end{equation}
we obtain 
\[
dT_1\wedge \cdots \wedge w_s=b'dT_1\wedge\cdots \wedge d\pi_\Ma\]
in $\hat{\Omega}_{\OK}^d(\OMa)$, and $b'\OMa=g'(\pi_\Ma)\OMa=D(\Ma/K_s)$; recalling that $l_f'(e_s)$ is a unit (since $l_f'(e_s)\equiv 1 \pmod{e_s}$ as $l_f'(X)=1+a_2X+a_3X^2+\cdots $, using in the notation of Section \ref{Description of the results}). This proves the first part of the statement.

For the second part, suppose that $\hat{g}(X)\in \mathcal{O}_{\Ma_{(0)}}[X]$ is another polynomial for which $\hat{g}(\pi_{\Ma})=e_s$. Then $\hat{g}(X)-g(X)$ is divisible by $r(X)$, the minimal polynomial of $\pi_\Ma$ over $\Ma_{(0)}$, i.e., 
$$\hat{g}(X)-g(X)=t(X)r(X)$$
 for some $t(X)\in \mathcal{O}_{\Ma_{(0)}}[X]$. Differentiating this and evaluating  at $X=\pi_{\Ma}$ we obtain 
 $\hat{g}'(\pi_\Ma)-g'(\pi_\Ma)=t(\pi_{\Ma})r'(\pi_{\Ma})$. Noticing that $r'(\pi_{\Ma})\OMa=D(\Ma/K)$ (which holds by an analogous argument as above, namely,   by observing this time that $\OMa=\mathcal{O}_{\Ma_{(0)}}[\pi_\Ma]$ and $D(\Ma/K)=D(\Ma/\Ka)=D(\Ma/\Ma_{(0)})$), we obtain 
 \[
\hat{g}'(\pi_{\Ma})\equiv g'(\pi_{\Ma}) \pmod{D(\Ma/K)} .
 \]
 The result now follows by multiplying both sides of this congruence by the unit $l_f'(e_s)$.
\end{proof}

\vspace{15pt}

We can now continue with the construction of $Q_{\Ma,s}$. Let us put 
\begin{equation}\label{qmsl}
Q_{\Ma,s}(T_1,\dots, T_{d-1},\pi_\Ma)=\frac{T_1\cdots T_{d-1}}{b'\pi^s}, 
\end{equation}
for $b'\in\OMa$ as in Lemma \ref{tech-lemma}, more specifically, $b'=l_f'(e_s)g'(\pi_{\Ma})$ as in Equation \eqref{myst-b}, where $g(X)\in \mathcal{O}_{\Ma_{(0)}}[X]$ is a polynomial such that $g(\pi_{\Ma})=e_s$. By  Proposition 4.1.5 (2) of \cite{Florez} we have that $D(K_s/K)=(\pi^s/\pi_1) \mathcal{O}_{K_s}$, 
then
\begin{equation*}\label{differentes}
\mathfrak{D}(\Ma/K)=\mathfrak{D}(\Ma/K_s)D(K_s/K)=\mathfrak{D}(\Ma/K_s)(\pi^s/\pi_1).
\end{equation*} 
 Moreover, we have $D(\Ma/K)\subset \mathfrak{D}(\Ma/K)$. This implies that $T_1\cdots T_{d-1}/(b'\pi^s)\in P_\Ma$ and $\mathfrak{D}(M/K)Q_{\Ma,s}(T_1,\dots, T_{d-1},\pi_\Ma)\equiv 0 \pmod {(\pi^s/\pi_1)P_\Ma}$. 
Therefore, by \cite{Florez} Proposition 4.2.3, $Q_{\Ma,s}$ 
defines  a $d$-dimensional derivation as follows
 \begin{equation}\label{general-derivation-def}
           Q_{\Ma,s}(a_1, \dots,a_d):= 
             \frac{\Tdc}{b'\pi^s}\,\det 
          \left[ \frac{\partial a_i}{\partial T_j} \right]_{_{1\leq i,j\leq d}}  ,
          \end{equation}
where $a_1,\dots, a_d\in \OMa$.
 Note that the definition of $Q_{\Ma,s}$ is independent of 
 the choice of uniformizer $\pi_\Ma$ of $\Ma$.


In the following two propositions we will show some properties of the derivations $Q_{\Ma,s}$
that will be needed later in the deduction of the main result.

\begin{prop}\label{extension-Q}
   Let  $\Na$ be a finite field extension of $\Ma$ such that $\Na\supset \Ka_t$, 
   for $t\geq s$, and having $T_1,\dots, T_{d-1},\pi_{\Na}$ as a system of local uniformizers.  
   Suppose that $D(\Na/\Ma)|\pi^{t-s}$, then $(\pi^t/\pi_1)P_\Na$
   is contained in $X:=(\pi^s/\pi_1)P_{\Ma}\mathcal{O}_{\Na}$ and so $Q_{\Na,t}\pmod{ X}$
   is well-defined. Let 
   \[\varpi: P_\Ma/(\pi^s/\pi_1)P_\Ma\to P_{\Na}/X\]
   be the injection induced from $P_\Ma\subset P_\Na$. Then the diagram
    \[
\begin{tikzcd}
 \OMa^d \arrow[r, " Q_{\Ma,s}"] \arrow[d, "Q_{\Na,t}\circ i"'] & \frac{P_\Ma}{(\pi^s/\pi_1)P_\Ma} \arrow[d, "\varpi"] \\
\frac{P_\Na}{(\pi^t/\pi_1)P_\Na }  \arrow[r, "\text{proj}"] & P_\Na/X
\end{tikzcd}\]
is commutative, where $i:\OMa^d\to \ONa^d$ is the inclusion map, and $\text{proj}$ is the projection map.
\end{prop}
\begin{proof}
The Proposition \ref{extension-Q} and its proof were suggested by  V. Kolyvagin.
First, 
\begin{align*}
\frac{\pi^t}{\pi_1}P_\Na=
\frac{\pi^t}{\pi_1}\frac{1}{\pi_1D(\Na/K)}&=
\frac{\pi^t}{\pi_1}\cdot\frac{1}{\pi_1D(\Na/\Ma)D(\Ma/K)}\\
&=\frac{\pi^s}{\pi_1}\cdot\frac{\pi^{t-s}}{D(\Na/\Ma)}.\frac{1}{\pi_1D(\Ma/K)}\subset X
\end{align*}
because $(\pi^{t-s}/D(\Na/\Ma))\subset \ONa$ by our assumption.

By Section \ref{Terminology and Notation} we can choose $a$ and $c$ in $\mathcal{O}_\Na$, and $b\in \mathcal{O}_\Ma$, such that
\begin{align*}
dT_1\wedge \cdots \wedge w_t &=c\,dT_1\wedge \cdots \wedge d\pi_\Na,\\
 dT_1\wedge \cdots \wedge w_s &=b\,dT_1\wedge \cdots \wedge d\pi_\Ma,\  \mathrm{and} \\
  dT_1\wedge \cdots \wedge d\pi_\Ma &=a\,dT_1\wedge \cdots \wedge d\pi_\Na.
\end{align*}
Since $\pi^{t-s}w_t=w_s$ by Proposition  \ref{generators}, then
\[
dT_1\wedge \cdots \wedge w_s=(\pi^{t-s}c)\,dT_1\wedge \cdots \wedge \pi_{\Na}=b\, dT_1\wedge \cdots \wedge d\pi_\Ma=(ab)\,dT_1\wedge \cdots \wedge d\pi_\Na.
\]

By virtue of Lemma \ref{tech-lemma} the elements $\pi^{t-s}c$ and $ab$  satisfy the congruence 
$$\pi^{t-s}c\equiv ba \pmod{D(\Na/K)}.$$
Dividing this congruence by $\pi^t\,c\,b$ and taking into 
account, again from Lemma \ref{tech-lemma}, that $c\,\mathcal{O}_\Na=D(\Na/K_t)$ and $b\,\mathcal{O}_\Ma=D(\Ma/K_s)$,
we obtain
\[
\frac{1}{\pi^sb}\equiv \frac{a}{\pi^tc} \pmod{Z},
\]
where 
\begin{align*}
Z=\frac{D(\Na/K)}{\pi^t\,D(\Na/K_t)\,D(\Ma/K_s)}
 =\frac{\ONa}{\pi_1\,D(\Ma/K_s)}
 =\frac{\pi^s}{\pi_1}\frac{\ONa}{\pi_1\,D(\Ma/K)}
 = X.
\end{align*}
Here we are using the fact that  $D(\Na/K)=D(\Na/K_t)D(K_t/K)$ and $D(K_t/K)=(\pi^t/\pi_1)\mathcal{O}_{K_t}$; similarly 
$D(\Ma/K)=(\pi^s/\pi_1)\,D(\Ma/K_s)$.
Thus 
\begin{align*}
Q_{\Na,t}(T_1,\dots, T_{d-1}, \pi_\Ma)
&=T_1\cdots T_{d-1}\frac{a}{\pi^tc}=T_1\cdots T_{d-1}\frac{1}{\pi^sb}\pmod{X}\\
&=\varpi\big(\,Q_{\Ma,s}(T_1,\dots, T_{d-1},\pi_\Ma)\,\big).
\end{align*}
This implies the corresponding equality for arbitrary $y\in (\OMa)^d$, i.e., 
\[
Q_{\Na,t}(y)=\varpi(Q_{\Ma,s}(y)) \pmod{X},
\]
 since the 
right-hand and left-hand mappings are multidimensional derivations of $(\OMa)^d$ over $\OK$ and thus
they are determined by their value at $(T_1,\dots,T_{d-1},\pi_\Ma)$ (cf. Proposition 4.2.3 of \cite{Florez}). 
\end{proof}

\begin{prop}\label{QGalois}
$Q_{\Ma,s}^g=\tau(g^{-1})Q_{\Ma,s}$ for all $g\in \normalfont \text{ Gal}(\overline{\Ka}/\Ka)$. Here $Q_{\Ma,s}^g$ 
is the $d$-dimensional derivation defined by 
$Q_{\Ma,s}^g(a_1,\dots, a_d):=[Q_{\Ma,s}(a_1^{g^{-1}},\dots,a_d^{g^{-1}})]^g$.
\end{prop}

\begin{proof}
Let $g\in \text{ Gal}(\overline{\Ka}/\Ka)$. Notice that it is enough to check that  
$$Q_{\Ma,s}^g(\Td,\pi_\Ma)=\tau(g^{-1})\,Q_{\Ma,s}(\Td,\pi_\Ma),$$ by Equation (41) of 
Proposition 4.2.3 in \cite{Florez}. 

Let $b'\in \mathcal{O}_\Ma$ such that $dT_1\wedge \cdots \wedge w_s=b'\, dT_1\wedge \cdots \wedge d\pi_\Ma$. Then 
\begin{align*}
dT_1\wedge \cdots \wedge w_s^g
&=dT_1^g\wedge \cdots \wedge w_s^g \hspace{40pt} (\text{Since $T_i^g=T_i$})\\
&=(dT_1\wedge \cdots \wedge w_s)^g \hspace{30pt}(\text{By Equation \eqref{action-diff} })\\
&=(b'\, dT_1\wedge \cdots \wedge d\pi_{\Ma})^{\,g}\\
&=(b')^g\, dT_1\wedge \cdots \wedge d\pi_{\Ma}^{\,g}  \hspace{13pt} (\text{By Equation \eqref{action-diff} }).
\end{align*}
By Proposition \ref{galois-action} this implies that $dT_1\wedge \cdots \wedge w_s=\tau(g^{-1})(b')^g\, dT_1\wedge \cdots \wedge d\pi_{\Ma}^g$.
Since $\pi_{\Ma}^g$ is also a uniformizer for $\Ma$ and  the definition of $Q_{\Ma,s}$
is independent of the uniformizer, then 
\[
Q_{\Ma,s}(\Td,\pi_{\Ma}^g)=\frac{T_1\cdots T_{d-1}}{\tau(g^{-1})\ (b')^g\ \pi^s}.
\]
But
\[
T_1\cdots T_{d-1} \frac{1}{\tau(g^{-1})(b')^g\pi^s}=T_1\cdots T_{d-1}\tau(g)\left( \frac{1}{b'\pi^s}\right) ^g,
\]
and this last expression is equal to $\tau(g)\,  Q_{M,s}(\Td,\pi_\Ma)^{\,g}$ by Equation (\ref{qmsl}).

\end{proof}

\subsection{The canonical logarithmic derivatives $QL_{\Ma,s}$}\label{The canonical logarithmic derivatives}
Let $\Ma$ be a finite extension of $\Ka_s$ with   $T_1,\dots,T_{d-1}$ and $\pi_\Ma$ as a system of local uniformizers. 

\begin{defin}\label{defilQL} We now construct what Kolyvagin  calls the logarithmic derivative associated to $Q_{\Ma,s}$ (cf. \cite{koly} page 333). This is the map
\[
QL_{\Ma,s}:K_d(\Ma)\longrightarrow \frac{\frac{1}{\pi_\Ma}P_\Ma}{\frac{\pi^s}{\pi_1\pi_\Ma}P_\Ma}
\]
defined by 
\begin{equation}\label{QL}
\left\{
\begin{aligned}
QL_{\Ma,s}(u_1,\dots,u_{d-1},\pi_\Ma)&=\frac{Q_{\Ma,s}(u_1,\dots, u_{d-1},\pi_\Ma)}{u_1\cdots u_{d-1}\pi_\Ma} \pmod{\frac{\pi^s}{\pi_1\pi_\Ma}P_\Ma},\\
QL_{\Ma,s}(u_1,\dots,u_d)&=\frac{Q_{\Ma,s}(u_1,\dots,u_d)}{u_1\cdots u_d} \pmod{\frac{\pi^s}{\pi_1}P_\Ma},\\
QL_{\Ma,s}(u_1,\dots,u_d\,\pi_\Ma^{\,k})&=k\, QL_{\Ma,s}(u_1,\dots,\pi_\Ma)+QL_{\Ma,s}(u_1,\dots,u_d),\ k\in\Z,\\
QL_{\Ma,s}(a_1,\dots, a_d)&=0, \ \text{whenever $a_i= a_j$ for $i\neq j$ and $a_1,\dots, a_d\in \Ma^*$}. 
\end{aligned}
\right.
\end{equation}
where $u_1,\dots, u_d$ are  in 
$\mathcal{O}_\Ma^*=\{x\in \mathcal{O}_\Ma:v_\Ma(x)=0\}$. 
\end{defin}
Notice that the forth
property says that $QL_{\Ma,s}$ is alternating, in particular
it is skew-symmetric, i.e.,
\[
QL_{\Ma,s}(a_1,\dots, a_i,\dots, a_j,\dots,a_d)=-QL_{\Ma,s}(a_1,\dots, a_j,\dots, a_i,\dots,a_d),
\]
whenever $i\neq j$.

From Equation \eqref{general-derivation-def}, applied to $\alpha=\{a_1,\dots, a_d\}\in K_d(\Ma)$ with $a_1,\dots, a_d\in \OMa$, the map $QL_{\Ma,s}$ can be put in explicit form as
\[
QL_{\Ma,s}(\alpha)=\frac{T_1\cdots T_{d-1}}{\pi^sb'}\cdot\frac{\det 
          \left[ \frac{\partial a_i}{\partial T_j} \right]_{_{1\leq i,j\leq d}} }{a_1\cdots a_d},
\]
where $b'=l_f'(e_s)g'(\pi_{\Ma})$ is as in Equation \eqref{myst-b}, where $g(X)\in \mathcal{O}_{\Ma_{(0)}}[X]$ is a polynomial such that $g(\pi_{\Ma})=e_s$. This helps explain the name \textit{logarithmic derivative}.

Notice also from Equation \eqref{QL} that $QL_{\Ma,s}$ is defined $\text{mod}\ \frac{\pi^s}{\pi_1}P_\Ma$
for elements $\alpha=\{u_1,\dots, u_d\}\in K_d(\Ma)$ with $u_1,\dots, u_d\in \OMa^*$. More specifically,
\begin{equation}\label{units-map-q}
QL_{\Ma,s}(u_1,\dots,u_d)=\frac{Q_{\Ma,s}(u_1,\dots,u_d)}{u_1\cdots u_d} \in P_{\Ma}/\left( \frac{\pi^s}{\pi_1}P_\Ma\right),
\end{equation}
which is a  fact that will be used later in Section \ref{The lift of $QL_n$}.

To state some of the properties of the map $QL_{\Ma,s}$ we need to introduce the following notation. 
Let us fix a $d$-dimensional local field $\La\supset K_n$. Let $\gamma_m$ be a uniformizer for 
$\La_m:=\La(\kappa_m)$, $m\geq 1$. From now on we will be using the following notation 
\begin{equation}\label{notational}
\begin{cases}
P_s=P_{\La_s} \quad (s\geq 1),\\ 
\ Q_s=Q_{\La_s,s},\ QL_s= QL_{\La_s,s}, \\
N_{s}=N_{\La_s/\La},\ N_{t/s}=N_{\La_t/\La_s} \quad  (t\geq s),\\
\text{Tr}_{s}=\text{Tr}_{\La_s/\La},\,\text{Tr}_{t/s}=\text{Tr}_{\La_t/\La_s}. 
\end{cases}
\end{equation}

Notice that
\begin{equation}\label{TrazaP_t}
\text{Tr}_{t/s}(P_t)\subset P_s \quad
\mathrm{and} \quad
\text{Tr}_{t/s}\big((1/\gamma_t)P_t\big)\subset (1/\gamma_s)P_s.
\end{equation}
Indeed, from $D(\La_t/K)=D(\La_t/\La_s)D(\La_s/K)$ we have that $P_t=(1/D(\La_t/\La_s))P_s$, thus
$$\text{Tr}_{t/s}\big(\,P_t\,\big)=\text{Tr}_{t/s}\big(\,(1/D(\La_t/\La_s))P_s\,\big)
=P_s\text{Tr}_{t/s}\big(\,1/D(\La_t/\La_s)\,\big)\subset P_s,$$
and $\text{Tr}_{t/s}(\big(\,\gamma_s/\gamma_t)P_t\,\big)\subset \text{Tr}_{t/s}(P_t)\subset P_s$.

Let $r'$ be smallest positive integer such that 
$\La\cap K_{\pi, \infty}\subset  K_{r'}$. The existence of such $r'$ can be guaranteed from the following argument. The field $\La\cap K_{\pi, \infty}$ is a finite extension of $K$ that is contained in $K_{\pi,\infty}$. Since $K$ is separable there exists an element $\alpha\in K_{\pi,\infty}$ such that 
$\La\cap K_{\pi, \infty}=K(\alpha)$. But since $K_{\pi,\infty}=\cup_{m\geq 1}K_{m}$, then $\alpha \in K_{m}$ for some $m\geq 1$. Thus $\La\cap K_{\pi, \infty}=K(\alpha)\subset K_m$ and therefore we can take  the smallest such $m$.

\begin{remark}
Notice that in general, however, $\La\cap K_{\pi, \infty}$ is not necessarily equal to $K_{r'}$. This can happen, for example, if the degree of inertia $f(K/\Qp)$  of $K/\Qp$ is greater than one. Indeed,
since $[K_{n+1}/K_n]=q$ for $q=|k_K|=p^{f(K/\Qp)}$ (cf. introduction of Section \ref{Canonical derivations and logarithmic derivatives} ), then
if $f(K/\Qp)\geq 2$ there exists an intermediate field $L$ between $K_{n+1}$ and $K_n$ such that $[L/K_n]=p$. This guarantees that $K_n\subsetneq L\subsetneq K_{n+1}$. Now let $\La=L\TT$. Then $\La\cap K_{\pi,\infty}=L$ and so $K_n\subsetneq \La\cap K_{\pi,\infty} \subsetneq K_{n+1}$.
\end{remark}

For $r'$ as above we have the identity
  \begin{equation}\label{Koly-f}
 \La_s\cap K_{\pi,\infty}=K_s \quad (s\geq r').
 \end{equation}
 \begin{proof}
 Indeed,
we will deduce this identity from the diagrams
 \[
 \begin{tikzpicture}[node distance = 1cm, auto]
      \node (Q) {$\La\cap K_{\pi,\infty}$};
      \node (E) [above of=Q, left of=Q] {$K_s$};
      \node (F) [above of=Q, right of=Q] {$\La$};
      \node (K) [above of=Q, node distance = 2cm] {$\La_s$};
      \draw[-] (Q) to node {} (E);
      \draw[-] (Q) to node {} (F);
      \draw[-] (E) to node {} (K);
      \draw[-] (F) to node {} (K);
      \end{tikzpicture}
      \hspace{30pt}
       \begin{tikzpicture}[node distance = 1cm, auto]
      \node (Q) {$\La\cap K_{\pi,\infty}$};
      \node (E) [above of=Q, left of=Q] {$\La_s\cap K_{\pi,\infty}$};
      \node (F) [above of=Q, right of=Q] {$\La$};
      \node (K) [above of=Q, node distance = 2cm] {$\La_s$};
      \draw[-] (Q) to node {} (E);
      \draw[-] (Q) to node {} (F);
      \draw[-] (E) to node {} (K);
      \draw[-] (F) to node {} (K);
      \end{tikzpicture}
      \]
       We begin by observing that from  $\kappa_s\subset \La_s$ and $\kappa_s\subset K_{\pi,\infty}$ we obtain
 \begin{equation}\label{finalle}
 K_s\subset \La_s \cap K_{\pi,\infty}.
 \end{equation}
      This  implies  $\La\cap K_s\subset \La\cap(\La_s\cap K_{\pi,\infty})\subset \La\cap K_{\pi,\infty}$. Since also $\La\cap K_{\pi,\infty}\subset K_{r'}\subset K_s$, by the definition of $r'$, then $\La\cap K_{\pi,\infty}\subset \La \cap K_s$. Thus
      \begin{equation}
      \label{ecua--1}
\La\cap K_s=\La\cap K_{\pi,\infty}
\quad     \mathrm{and}\quad
\La\cap (\La_s\cap K_{\pi,\infty})=\La\cap K_{\pi,\infty}.  
            \end{equation}
      On the other hand, since $\kappa_s\subset K_s$ and $\kappa_s\subset \La_s\cap K_{\pi,\infty}$ then
      \begin{equation}
      \label{ecua--2}
\La\cdot K_s=\La_s 
\quad     \mathrm{and}\quad
\La\cdot (\La_s\cap K_{\pi,\infty})=\La_s .     
            \end{equation}
      Therefore \eqref{ecua--1} and \eqref{ecua--2}, bearing in mind the above diagrams, imply that  the maps
      \[
      \mathrm{Gal}(\La_s/\La)\longrightarrow \mathrm{Gal}(K_s\,/\,\La\cap K_{\pi,\infty}) 
      \]
      and
        \[
      \mathrm{Gal}(\La_s/\La)\longrightarrow \mathrm{Gal}(\La_s\cap K_{\pi,\infty}\,/\,\La\cap K_{\pi,\infty})
      \]
     are isomorphisms (cf. \cite{Lang} VI \S 1 Theorem 1.12). In particular
     \[
[\La_s:\La]
=[K_s\,:\,\La\cap K_{\pi,\infty}]
= [\La_s\cap K_{\pi,\infty}\,:\,\La\cap K_{\pi,\infty}].    
     \]
     This combined with \eqref{finalle} yields \eqref{Koly-f}.
 \end{proof}

The following proposition inform us about the compatibility of the maps $QL_s$ and $QL_t$ for $t\geq s\geq r'$.
\begin{prop}\label{levelQL}
For any $t\geq s \geq r'$ the following diagram commutes
\[
\begin{CD}
 K_d(\La_t) @>{ QL_{t} }>> \frac{\frac{1}{\gamma_t}P_t}{\frac{\pi^t}{\pi_1\gamma_t}P_t}   \\
@V{N_{t/s}}VV @VV{\normalfont \text{Tr}_{t/s}}V \\
K_d(\La_s) @>>{QL_s}> \frac{\frac{1}{\gamma_s}P_s}{\frac{\pi^s}{\pi_1\gamma_s}P_s}.
\end{CD}
\]
\end{prop}

\begin{proof}[Proof of Proposition \ref{levelQL}]
The proof follows the same ideas as  in \cite{koly} Proposition 7.13.    It will be enough to consider the case $t=s+1$. Before we continue with the proof we need the following lemma.
\begin{lemma}\label{important}
Let $\Na/\Ma$ be a finite extension of $d$-dimensional local fields such that 
$\La_{s}\subset \Ma \subset \Na\subset  \La_{t}$. Then
\[
QL_t\big(\,N_{\Ma/\Na}(\alpha)\,\big)=
\left(\sum \tau_t(g_i)g_i \right)QL_t(\alpha) \quad (\,\forall \alpha\in K_d(\Na)\,),
\]
where the sum runs through any  collection of  elements
 $g_i\in \normalfont \text{Gal}(\La_t/\La_s)$ that form a full set of
 representatives of the group $\normalfont \text{Gal}(\Na/\Ma)$  $\normalfont (\, \simeq \text{Gal}(\La_t/\Ma)/\text{Gal}(\La_t/\Na)\,)$. In particular,
\begin{equation}\label{crux-prop}
QL_t\big(\,N_{t/s}(\alpha)\,\big)=
\left(\sum_{g\in \normalfont \text{Gal}(\La_t/\La_s)} \tau_t(g)g \right)QL_t(\alpha) 
\quad (\, \forall \alpha\in K_d(\La_t)\,).
\end{equation}
\end{lemma}
\begin{proof}
Since $[\La_{t}:\La_s]$ is a power of the prime $p$, then $[\Na:\Ma]=p^m$ for some $m\geq 1$. We will do the proof by induction on $m$. For this reason assume first that $[\Na:\Ma]=p$. In this case, $K_d(\Na)$ is generated by all the elements of the form 
 $\alpha=\{a_1,\dots, a_d\}$ with $a_1\in \Na^*$ and $a_2,\dots, a_d\in \Ma^*$ (cf. \cite{Zhukov} Chapter 9, Section 2.5, Corollary 1 and Corollary 2, or \cite{Morrow} Lemma 2.1 ). Thus it is enough to check the identity for these elements only.  Therefore let $\alpha=\{a_1,\dots, a_d\}$ be of said form and let $g_i$ be as in the statement of the lemma. Under these conditions we easily verify the identity
 \begin{align*}
QL_t(N_{\Na/\Ma}\{a_1,\dots a_d\})
&=QL_t(\{N_{\Na/\Ma}(a_1),a_2,\dots, a_d\}) \quad \text{(by \cite{Florez} Proposition 2.1.2 (4))}\\ 
&=QL_t(\{\prod a_1^{g_i},\dots, a_d\})\\
&=\sum QL_t(\{ a_1^{g_i},\dots, a_d\})\\
&=\sum QL_t(\{ a_1^{g_i},\dots, a_d^{g_i}\})\quad \text{($a_j^{g_i}=a_j$, $2\leq j \leq d$)}\\
&=(\,\sum \tau_t(g_i)g_i\,)\,QL_t(\{a_1,\dots, a_d\}) \quad \text{(by Proposition \ref{QGalois})},
 \end{align*}
 where $\Sigma$ runs through the elements $g_i$.
 
 Suppose now that $[\Na:\Ma]=p^m$, for $m\geq 2$, and assume that the identity is true for any extension of degree $p^w$ with $w<m$. Let $\mathcal{E}$ be a $d$-dimensional local field such that 
 $\Ma\subset \mathcal{E}\subset \Na$, 
 $[\Na:\mathcal{E}]=p^{m-1}$ and $[\mathcal{E}:\Ma]=p$. Let $g_i$ and  $h_j$ be elements of $\text{Gal}(\La_t/\La_s)$ such that $g_i$, $h_j$ and $g_i\cdot h_j$, respectively, run through a system of representatives of the groups $\text{Gal}(\Na/\mathcal{E})$, $\text{Gal}(\mathcal{E}/\Ma)$ and 
 $\text{Gal}(\Na/\Ma)$, respectively. Then for any $\alpha\in K_d(\Na)$ we have
 \begin{align*}
 QL_t(N_{\Na/\Ma}(\alpha))
 &= QL_t\big(\,N_{\mathcal{E}/\Ma}\big(N_{\Na/\mathcal{E}}(\alpha)\big)\,\big)\\
 &=(\sum \tau_t(h_j)h_j)QL_t(N_{\Na/\mathcal{E}}(\alpha)) \quad \text{(By induction hypothesis)}\\
&=(\sum \tau_t(h_j)h_j) (\sum \tau_t(g_i)g_i)) QL_t(\alpha)\quad \text{(By induction hypothesis)} \\
&=(\sum \tau_t(g_ih_j)g_ih_j) QL_t(\alpha),
 \end{align*}
 which verifies the identity in this general case as well.
\end{proof}

    We continue with the proof of Lemma \ref{levelQL}. From Equation \eqref{crux-prop} it follows that
    for all $\alpha\in K_d(\La_t)$
          \begin{multline}\label{q-trace-normal}
          QL_t\left(N_{t/s}(\alpha)\right)=\left( \sum \, \tau_t(g)g\, 
                                               \right)QL_t(\alpha)=\\
                                          =\left( \sum g\right)QL_t(\alpha)
                                            + \left( \sum \big(\tau_t(g)-1\big)g \right) QL_t(\alpha),
     \end{multline}
     where each $\sum$ is taken over all $g\in \text{Gal}(\La_t/\La_s)$.
    By Proposition \ref{imagedelta} below and by Equation (\ref{Koly-f}), we see that $\sum (\tau_t(g)-1)g$ 
     takes 
     \[P_t=\frac{1}{D(\La_t/\La_s)}P_s=\frac{\pi}{D(\La_t/\La_s)}\left( \frac{1}{\pi}P_s\right) 
     \]
      to
     \[
     \left( \frac{\pi^{t}}{\pi_1}\OZ_{\La_t}\right) \ \left( \frac{1}{\pi}P_s\right) =\frac{\pi^s}{\pi_1}P_s\OZ_{\La_t}.
     \]
     Then
     \begin{align*}
     QL_s\big(\,N_{t/s}(\alpha) \,\big)&=QL_t\big(\,N_{t/s}(\alpha)\,\big) \pmod{ \frac{\pi^s}{\pi_1\gamma_s}P_s}\\
                       &=\text{Tr}_{t/s}\big(\,QL_t( \alpha)\,\big),     
     \end{align*}
          where the first equality follows from Proposition \ref{extension-Q} and the second from Equation \eqref{q-trace-normal}.
\end{proof}

\begin{prop}\label{imagedelta}
      Let $\Ma\supset K$ be a $d$-dimensional local field such that 
      $\Ma\cap K_{\pi,\infty}=K_s$ and let $\Na=\Ma_{s+1}$. 
       Then
            \[
   \normalfont   \tau_{s+1}(\text{Gal}(\mathcal{N}/\Ma))=
      \frac{1+\pi^sC}{1+\pi^{s+1}C} \subset (\,C/\pi^{s+1}C\,)^*,
      \]
                where $C=\OK$, and the element
            \[
   \normalfont   \sum_{g\in \text{Gal}(\mathcal{N}/\Ma)}(\tau_{s+1}(g)-1)g 
      \]
           takes $(\pi/D(\Na/\Ma))\mathcal{O}_{\mathcal{N}}$ to $(\pi^{s+1}/\pi_1)\mathcal{O}_{\mathcal{N}}$. Also $D(\Na/\Ma)|\pi$.
      
\end{prop}

\begin{proof}

This follows immediately from the fact 
that $\text{Gal}(\Na/\Ma)\cong \text{Gal}(K_{s+1}/K_{s})$, $D(\Na/\Ma)=D(K_{s+1}/K_s)\ONa$, 
and Proposition 7.12 in \cite{koly} and its proof.

\end{proof}

\section{Generalized Artin-Hasse formulas }\label{Computations of the invariants}
In this section we will give a version of Artin-Hasse formula for the Kummer pairing over higher local fields. In the next section, we will use these results to deduce the main formulas.

We will assume that  $p\neq 2$. Let $\varrho$ be the ramification of index of $K/\Qp$.
Let $(k,t)$ 
be a pair of integers for which there exist a positive integer $m$
such that $t-k-1\geq m\varrho \geq k$, 
in other words: $(k,t)$ is an \textit{admissible pair} as in Definition 2.3.1 of \cite{Florez}.
Also, recall that $\Ka_t=K_t\TT$ and  let $ \Ma$ be a finite extension of $\Ka_t$  having   $T_1,\dots,T_{d-1}$ and $\pi_\Ma$ as a system of local uniformizers.  We define
\[
R_{\Ma,1}:=
\big\{x\in \Ma: v_{\Ma}(x)\geq -v_{\Ma}(D(M/K))-\floor*{\frac{v_{\Ma}(p)}{p-1}}-1\big\}
\]
and
\[
V_{\Ma,1}:=\{u\in \OMa:v_{\Ma}(u-1)>v_{\Ma}(p)/(p-1)\}.
\]

By Equation (49)  of Definition 5.2.1 of \cite{Florez} there exists a 
 $c_{\beta:1}\in R_{\Ka_t,1}/\pi^kR_{\Ka_t,1}$ such that
\begin{equation}\label{defi-inv}
\big(\{\Td,u\},\ e_t\big)_{\Ka_t,k}^1
=\mathbb{T}_{\Ka_t/K} \big( \, \log( u) \, c_{\beta:1}\,  \big) \pmod{\pi^kC},
\end{equation}
for all $u\in V_{\Ka_t,1}$, where $\log$ is the usual logarithm: $\log(X)=(X-1)-(X-1)^2/2+\cdots$. Here it is useful to recall the notation from \eqref{notation-kummer-pairing} for the Kummer pairing $\big(\{\Td,u\},\ e_t\big)_{\Ka_t,k}$, i.e., the element $\big(\{\Td,u\},\ e_t\big)_{\Ka_t,k}^1 \in C/\pi^kC$ is such that
\[
\big(\{\Td,u\},\ e_t\big)_{\Ka_t,k}=\left[ \big(\{\Td,u\},\ e_t\big)_{\Ka_t,k}^1\right]_f(e_k).
\]
For a finite extension $\Ma/\Ka_t$ we denote by $\overline{c}_{\beta:1}$ the image of $c_{\beta:1}$ under the map
$R_{\Ka_t,1}/\pi^kR_{\Ka,t}\to R_{\Ma,1}/\pi^kR_{\Ma,1}$. For this element we have again, by (49)  of Definition 5.2.1 of \cite{Florez}, that
\begin{equation}\label{defi-invi-general}
\big(\{\Td,u\},\ e_t\big)_{\Ma,k}^1
=\TrM \big( \, \log( u) \, \overline{c}_{\beta:1}\,  \big) \pmod{\pi^kC},
\end{equation}
for all $u\in V_{\Ma,1}$.

In the following proposition we will  compute the constant $c_{\beta:1}$ explicitly. More specifically, we will show that
 \[
 c_{\beta:1}=-1/\pi^t \pmod{\pi^k R_{\Ka_t,1}},
 \]
which implies 
\begin{equation}\label{inv-form}
 \overline{c}_{\beta:1}=-1/\pi^t \pmod{\pi^k R_{\Ma,1}}.
  \end{equation}

Before we prove this we have to make one final observation. As it is remarked in \cite{Florez} ( right after Definition 5.2.1), the $c_{\beta:1}$ is an invariant of the isomorphism class of $F$. More concretely, if $\phi:(F,e_n)\to (\tilde{F},\tilde{e}_n)$ is an isomorphism of formal groups over $C$ such that $\phi(e_n)=\tilde{e}_n$ for all $n\geq 1$, and we denote the corresponding constant associated to $\tilde{F}$ by $\tilde{c}_{\beta:1}$, then $c_{\beta:1}=\tilde{c}_{\beta:1}\pmod{\pi^kR_{\Ka_t,1}}$. This is also pointed out, in the one-dimensional case, in \cite{koly} Equation c1 page 321. This fact can be used to our advantage in the following way. We may assume, by going to an isomorphic formal group law, that $r(X)=X$ is a $k$-normalized series of $F$ (cf. \cite{Florez} Section 2.4); this is a standard trick used by Kolyvagin in \cite{koly}. This means that for every $d$-dimensional local field $\Ma$ containing $\kappa_{f,k}$  (which will be our case since we are assuming $\Ma\supset K_t$ and $t>k$ ), the Kummer pairing satisfies 
\begin{equation}\label{norm-series-identity}
(\alpha,\ x)_{\Ma,k}=0
\end{equation}
for every $\alpha=\{a_1,\dots, a_d\}\in K_d(\Ma)$ such that $a_i=x$ for some $1\leq i \leq d$. 

We are now ready to state and prove the main result in this section.
  
\begin{prop}\label{Artin-Hasse}
Let $\Ma=\Ka_t$. For all $u\in V_{\Ma,1}$ we have
\begin{equation}
\big(\{\Td,u\},\ e_t\big)_{\Ma,k}^1
=\TrM \left(  \log (u )\, \left(-\frac{1}{\pi^t} \right)  \right) \pmod{\pi^kC},
\end{equation}
or equivalently,
\begin{equation}\label{compinv}
\normalfont \big(\{\Td,u\},\ e_t\big)_{\Ma,k}^1
=\text{Tr}_{K_t/K}\left( c_{\Ma/K_t}( \log (u)) \, \left(-\frac{1}{\pi^t} \right)  \right) \pmod{\pi^kC},
\end{equation}
where $c_{\Ma/K_t}$ is defined in Section \ref{Terminology and Notation}. In particular,
\begin{equation*}\label{}
 c_{\beta:1}=-1/\pi^t \pmod{\pi^k R_{\Ma,1}}.
  \end{equation*}
\end{prop}
\begin{proof}
Since $K_t/K$ is a totally ramified extension we will take as a uniformizer $\pi_t$ of $K_t$ (and consequently of $\Ma=\Ka_t$) the torsion point $e_t$. We start by observing that  every $u\in V_{\Ma,1}$ can be expressed as 
\[
 \prod_{\substack{\bar{i}
 =\id\in S\\ i_d\geq [v_{M}(p)/(p-1)]+1}}
 (1+\theta_{\,\bar{i}}\,T_1^{i_1} \cdots T_{d-1}^{i_{d-1}}\pi_t^{i_d}),
\]
where  $\theta_{\bar{i}}\in \mathcal{R}$, $\mathcal{R}$ is 
the group of $(q-1)$th roots of 1 in $K_t^*$, 
and $S\subset \Z^d$ is an admissible set ( cf. Corollary from Section 1.4.3 of  \cite{Zhukov}). The convergence of this product is with respect to the Parshin topology of $\Ma^*$ (cf. \cite{Zhukov} Chapter 1 \S1.4.2.). On the other hand, by Proposition 2.2.1 (6) of \cite{Florez} the mapping 
\[
V_{\Ma,1}\longrightarrow C/\pi^kC: u\to (\{T_1,\dots, T_{d-1},u\},e_t)_{\Ma,k}^1
\]
 is sequentially continuous with respect to the Parshin topology of $V_{\Ma,1}\subset \Ma^*$. Similarly the mapping
 \[
 V_{\Ma,1}\longrightarrow C/\pi^kC\,:\, u\mapsto  \TrM \left(  \log u \, \left(-\frac{1}{\pi^t} \right)  \right)
 \]
 is also sequentially  continuous with respect to the Parhsin topology by the following facts:
 $\log$ is sequentially continuous (cf. \cite{Florez} Lemma 6.2.3 and Remark 6.2.2), multiplication by a  fixed constant is a sequentially continuous map (cf. \cite{Florez} Proposition 6.1.1.), and $\TrM$ is also sequentially continuous (cf. \cite{Florez} Section 3.1). 
 
 With these considerations
it is enough to check (\ref{compinv}) for 
$$u=1+\theta\, T_1^{i_1} \cdots T_{d-1}^{i_{d-1}}\pi_t^{i_d}, \quad (\theta^{q-1}=1).$$
We are going to do this by cases:

Case 1) Suppose $(i_1\dots, i_{d-1})\neq (0,\dots, 0)$. Then the right hand 
side of (\ref{compinv}) is zero since $c_{\Ma/K_t} (\log u)=0$.  
Let us show that the left hand side is also zero as well. 

Suppose first $i_r>0$ for some $1\leq r\leq d-1$. Consider $\mathcal{N}=K_t\{\{Y_1\}\}\dots \{\{Y_{d-1}\}\}$ where $Y_r=T_r^{i_r}$ and $Y_m=T_m$, for $m\neq r$.
By lemma \ref{kummerian} below, $\Ma/\mathcal{N}$ is a 
finite extension of degree $i_r$ and 
$N_{\Ma/\mathcal{N}}(T_r)=\pm Y_r$.  Let also $i'_r=1$ and $i'_m=i_m$ for $m\neq r$.
To simplify the notation we will denote $T_1^{i_1}\cdots T_{d-1}^{i_{d-1}} $ and $Y_1^{i'_1}\cdots  Y_{d-1}^{i'_{d-1}}$ by $T^{\bar{i}}$ and $Y^{\bar{i'}}$, respectively.  Therefore by  \cite{Florez} Proposition  2.2.1 (4) 
\begin{align*}
&\left(\{\Td,1+\theta\, T^{\bar{i}}\ \pi_t^{i_d}\},\ e_t\right)_{\Ma,k}=\\
&=\left(N_{\Ma/\mathcal{N}}\{\Td,1+\theta\, T^{\bar{i}}\ \pi_t^{i_d}\},\  
  e_t\right)_{\mathcal{N},k}\\
&=\left(\{Y_1,\dots,N_{\Ma/\mathcal{N}}(T_r),\dots ,Y_{d-1},\ 1+\theta\, Y^{\bar{i'}}\pi_t^{i_d}\},\ e_t\right)_{\mathcal{N},k}\\
&=\left(\{Y_1,\dots, \pm1,\dots,Y_{d-1},1+\theta\,  Y^{\bar{i'}}\, \pi_t^{i_d}\}, \ e_t\right)_{\mathcal{N},k}\\
 &\hspace{150pt}\oplus \left(\{Y_1,\dots,Y_{d-1},1+\theta\, Y^{\bar{i'}}\pi_t^{i_d}\},\ e_t\right)_{\mathcal{N},k}.
\end{align*}
Since $p\neq 2$, $\left(\{Y_1,\dots,\pm1,\dots,Y_{d-1},1+\theta\, Y^{\bar{i'}}\ \pi_t^{i_d}\},\ e_t\right)_{\mathcal{N},k}=0$. On the other hand, since $\theta ^{q-1}=1$ then  
\begin{align*}
&\left(\{Y_1,\dots,\theta,\dots,Y_{d-1},1+\theta\, Y^{\bar{i'}}\pi_t^{i_d}\},\ e_t\right)_{\mathcal{N},k}\\
&=\frac{1}{q-1}\left(\{Y_1,\dots,\theta^{q-1},\dots,Y_{d-1},1+\theta\, Y^{\bar{i'}}\pi_t^{i_d}\},\ e_t\right)_{\mathcal{N},k}=0 
\end{align*}
( multiplication by $1/(q-1)$ makes sense since $(q-1,p)=1$). Moreover,
$$\left(\{Y_1,\dots,\pi_t,\dots,Y_{d-1},1+\theta\, Y^{\bar{i'}}\pi_t^{i_d}\},\ e_t\right)_{\mathcal{N},k}=0$$ from the relation \eqref{norm-series-identity} applied to the field $\Na\supset K_t$ and the Kummer pairing $(,)_{\Na,k}$, namely,
$(\{a_1,\dots,x,\dots, a_d\},\ x)_{\mathcal{N},k}=0$ for all $\{a_1,\dots,x,\dots, a_d\}\in K_d(\Na)$ ( recalling that $\pi_t=e_t$). Thus
\begin{align*}
&\left(\{\Td,1+\theta\, T^{\bar{i}}\pi_t^{i_d}\},\ e_t\right)_{\Ma,k}\\
=&\left(\{Y_1,\dots,Y_r,\dots,Y_{d-1},1+\theta\, Y^{\bar{i'}}\pi_t^{i_d}\},\ e_t\right)_{\mathcal{N},k}\\
=&\left(\{Y_1,\dots, \theta\, Y_1^{i'_1}\cdots Y_r^{i'_r}\cdots Y_{d-1}^{i'_{d-1}}\pi_t^{i_d},\dots,Y_{d-1},
1+\theta\, Y^{\bar{i'}}\pi_t^{i_d}\},\ e_t\right)_{\mathcal{N},k}\\
=&\left(\{Y_1,\dots, -\theta\, Y^{\bar{i'}}\pi_t^{i_d},\dots,Y_{d-1},1+\theta\, Y^{\bar{i'}}\pi_t^{i_d}\},\ e_t\right)_{\mathcal{N},k}.
\end{align*}
The second equality follows from the fact that 
$\{Y_1,\dots, Y_m,\dots, Y_{d-1},1+\theta\, Y^{\bar{i'}}\pi_t^{i_d}\}$ 
is trivial, for $m\neq r$, in the Milnor K-group $K_d(\Ma)$. 
Moreover, the last expression in the chain of equalities is again zero because 
$\{Y_1,\dots, -\theta\, Y^{\bar{i'}}\pi_t^{i_d},\dots,Y_{d-1},1+\theta\, Y^{\bar{i'}}\pi_t^{i_d}\}$
is the zero element, by the Steinberg property, in the Milnor K-group $K_d(\Ma)$.

Suppose now $i_r<0$. We take $Y_r=T_k^{-i_r}$ instead and 
by lemma \ref{kummerian} we have 
$N_{\Ma/\mathcal{N}}(T_k^{-1})=\pm  T_r^{-i_r}=\pm Y_r$. 
Noticing that 
\begin{align*}
&\left(\{T_1,\dots, T_r,\dots, T_{d-1},1+\theta\, T^{\bar{i}}\pi_t^{i_d}\},\ e_t\right)_{\Ma,k}\\
&=-\left(\{T_1,\dots, T_r^{-1},\dots, T_{d-1},1+\theta\, T^{\bar{i}}\pi_t^{i_d}\},\ e_t\right)_{\Ma,k}
\end{align*}
we can now apply the same argument as before to conclude that 
\[\left(\{T_1,\dots, T_r,\dots, T_{d-1},1+\theta\, T^{\bar{i}}\pi_t^{i_d}\},\ e_t\right)_{\Ma,k}=0.\]

Case 2) Suppose $(i_1\dots, i_{d-1})=(0,\dots, 0)$. 
That is, when $u$ is an element of the one dimensional 
local field $K_t$. In this case, we will show in lemma 
\ref{one-two-connec} below that the pairing 
$\left(\{T_1\dots, T_{d-1},u\},\ e_t\right)_{\Ma,k}$ 
coincides with the pairing taking values in the 
one dimensional local field $K_t$, namely  
$\left(u,\ e_t\right)_{K_t,k}$. Thus, by \cite{koly} 
section 7.3.1 and the fact that $c_{\Ma/K}(\log (u))=\log(u)$ formula (\ref{compinv}) follows.
\end{proof}

\begin{lemma}\label{kummerian}
Let $M$ be a complete discrete valuation field and $\Ma=M\T$. 
Put $Y=T^{j}$ for $j> 0$. Define $\mathcal{N}=M\{\{Y\}\}$.
Then $\Ma/\mathcal{N}$ is a finite extension of degree $j$ and 
$N_{\Ma/\mathcal{N}}(T)=\pm Y$.
\end{lemma}
\textbf{Remark}: Since $M$ is a complete discrete valuation field the result 
 immediately generalizes to the $d$-dimensional case, for if $\La=L\TT$, then we can take
 $M$ to be $L\{\{T_1\}\}\dots \{\{T_{d-2}\}\}$   and apply the result to $\Ma=M\{\{T_{d-1}\}\}$
 and $\mathcal{N}=M\{\{Y_
 {d-1}\}\}$, where $Y_{d-1}=T_{d-1}^j$.
\begin{proof}
We can assume that $M$ contains $\zeta_j$, 
a primitive $j$th root  of 1. 
Otherwise we can consider the diagram
\[
 \begin{tikzpicture}[node distance = 2cm, auto]
      \node (Q) {$\mathcal{N}=M\{\{Y\}\}$};
      \node (E) [above of=Q, left of=Q] {$\Ma=M\T$};
      \node (F) [above of=Q, right of=Q] {$M(\zeta_j)\{\{Y\}\}$};
      \node (K) [above of=Q, node distance = 4cm] {$M(\zeta_j)\T$};
      \draw[-] (Q) to node {} (E);
      \draw[-] (Q) to node {} (F);
      \draw[-] (E) to node {} (K);
      \draw[-] (F) to node {} (K);
      \end{tikzpicture}
      \]
we see that $[\Ma/\mathcal{N}]=[M(\zeta_j)\T: M(\zeta_j)\{\{Y\}\}]$ 
since $M\T \cap\left(  M(\zeta_j)\T \right) =M\{\{Y\}\}$. 

Note that $Y$ has exact order $j$ in $\mathcal{N}^*/(\mathcal{N}^*)^j$ for if 
$Y=\alpha^k$, $\alpha\in \mathcal{N}^*$, then 
$0=v_{\mathcal{N}}(Y)=kv_{\mathcal{N}}(\alpha)$, thus 
$\alpha \in \mathcal{O}_{\mathcal{N}}^*$, and we can go 
to the residue field $k_{\mathcal{N}}$ of $\mathcal{N}$ where we have 
$1=v_{k_\mathcal{N}}(\overline{Y})=kv_{k_\mathcal{N}}(\overline{\alpha})$, which implies $k=1$.   
Then by Kummer theory (cf. \cite{cassels} Chapter 3 Lemma 2) 
we have that the polynomial $P(X)=X^j-Y\in \mathcal{O}_{\mathcal{N}}[X]$   
is irreducible.  Thus  $[\Ma/\mathcal{N}]=j$, 
and $N_{\Ma/\mathcal{N}}(T)$ is the product 
of the roots of the polynomial $P(X)$. 
These roots are $\zeta_j^kT$, $k=1,\dots,j$. Thus
\[
N_{\Ma/\mathcal{N}}(T)
=\prod_{k=1}^{j}\zeta_j^kT
=\zeta_j^{\frac{j(j+1)}{2}}T^j
=
\begin{cases}
T^j=Y, &  \text{if $j$ is odd,}\\
-T^j=-Y, &  \text{if $j$ is even.}\\
\end{cases}
=(-1)^{j+1}Y.
\]
\end{proof}


\begin{lemma}\label{one-two-connec} 
For a local field $L/\Qp$, let $\La=\LTT$ and define the map
\begin{align}
f: L^*
\to \normalfont \text{Gal}(L^{ab}/L)
\quad \mathrm{ 
 by} \quad
a\to \Upsilon_{\La} (\{\Td,a\})\bigg|_{_{\text{Gal}(L^{ab}/L)}},
\end{align}
induced by the natural restriction $\normalfont \text{Gal}(\La^{ab}/\La)\to \text{Gal}(L^{ab}/L)$. Recall that $\Upsilon_\La:K_d(\La)\to \text{Gal}(\La^{ab}/\La)$  is  
Kato's reciprocity map for $\La$ (cf. \cite{Florez} \S2.1.4). Then $f$  coincides with the Artin's local reciprocity map for $L$:  
$\theta_L:L^*\to \normalfont \text{Gal}(L^{ab}/L)$. 
Thus, for 
$L=K_t$  and $\Ma=K_t\TT$ we have 
\[\left(\{\Td,u\},\ e_t\right)_{\Ma,k}=\left(u,\ e_t\right)_{K_t,k},\]
for all $u\in V_{L,1}=\{x\in L: v_L(x-1)>v_L(p)/(p-1)\}$. Here $\left(u,\ e_t\right)_{K_t,k}$
 denotes the Kummer pairing of the (one-dimensional) local field $K_t$.
\end{lemma}
\begin{proof}  
It is enough to verify the two conditions of \cite{cassels} 
Chapter 5 \S 2.8 Proposition 6. Let $\La_{(d)}$ be $\La$, and 
$\La_{(d-1)}=k_{L}((t_1))\cdots((t_{d-1})),\dots, \La_{(0)}=k_{L}$  
the chain of residue fields of $\La$. 
Let $\partial :K_m(\La_{(m)})\to K_{m-1}(\La_{(m-1)})$, $m=1,\dots, d$, be the boundary map defined in Chapter 6, Section 6.4.1 of \cite{Zhukov} ( See also  \cite{Fesenko} IX \S2). Then by Theorem 3 of Section 10.5 of \cite{Zhukov}
the following diagram commutes:
\[
\begin{CD}
 K_d\left(\La_{(d)}\right) 
           @>{\Upsilon_{ \La_{(d)} }}>> 
                 \normalfont \text{Gal}\big(\,\La_{(d)}^\text{ab}/\La_{(d)}\,\big)  \\
@V{\partial}VV @VV{\sigma\to \overline{\sigma}}V \\
 K_{d-1}\left( \La_{(d-1)}\right)  
         @>{\Upsilon_{\La_{(d-1)}}}>> 
              \normalfont \text{Gal}\big(\,\La_{(d-1)}^{\text{ab}}/\La_{(d-1)}\,\big)\\ 
\vdots   @. \vdots \\
  @V{\partial}VV @VV{\sigma\to \overline{\sigma}}V \\
 \mathbb{Z}=K_0\left(\La_{(0)} \right)  
        @>{\Upsilon_{\La_{(0)}}}>> \normalfont \text{Gal}(\La_{(0)}^{\text{ab}}/\La_{(0)})
\end{CD}
\]
Moreover, by Section 6.4.1 of \cite{Zhukov} (See also \cite{Kato}  Section 1, Theorem 2) 
the composition of the vertical maps $\partial\circ \cdots \circ \partial( \{T_1\dots, T_{d-1},a\})$ coincide with the valuation $v_L(a)$ for $a\in L^*$. 
Therefore $f:L^*\to \text{Gal}(L^{ab}/L)\to \text{Gal}(L^{un}/L)$ 
is the valuation map $v_L:L^*\to \mathbb{Z}$. 
Thus condition (1) of \cite{cassels} Chapter 5 \S 2.8 
Proposition 6 is verified.

 Suppose $L'/L$ is a finite abelian extension and let $\La'=L'\TT$.
 If $a\in L^*$ 
is a norm from $(L')^*$, 
namely $a=N_{L'/L}(\alpha)$, then clearly $\{T_1,\dots, T_{d-1},a\}$ 
is a norm from $K_d(\La')$, 
namely 
\[\{T_1,\dots, T_{d-1},a\}=\{T_1,\dots, T_{d-1},N_{L'/L}(\alpha)\}=N_{\La'/\La} \{T_1,\dots, T_{d-1},\alpha\},.\] 
Then by (1) of Theorem 2.1.1 of \cite{Florez}  we have that $\Upsilon_{\La} ({T_1,\dots,T_{d-1},a}) $ is trivial 
on $\La'$ and so $f(a)$ is trivial on $L'$. 
Thus condition (2) of \cite{cassels} Chapter 5 \S 2.8 
Proposition 6 is verified.
\end{proof}

\section{Generalized Kolyvagin formulas}\label{Generalized Kolyvagin formulas}

In this section we will provide a refinement 
of the formulas given for Theorem 5.3.1. of \cite{Florez} 
to the case of a Lubin-Tate formal group $F=F_f$. 

Using the notation in (\ref{notational}) we define
 \[
 K_d(\La_s)'=\bigcap_{t\geq s}N_{t/s}\big(\, K_d(\La_t)\,\big).
 \]
Also, as in Section \ref{The canonical logarithmic derivatives},  let $r$ maximal, and $r'$ minimal, such that
\[
K_r\subset \La\cap K_{\pi,\infty}\subset K_{r'}.
\] 

Recall from Section \ref{Description of the results} that $F(\mu_{\La})$ denotes the ideal $\mu_{\La}=\{x\in\La:v_{\La}(x)>0\}$ of $\OLa$ endowed with the group structure of $F$. Let $T_{\La}$ be the image of the map $l_f:F(\mu_{\La})\to \La$. Then clearly $T_{\La}$ is a $C$-module. Let $R_{\La}$ be the dual of $T_{\La}$ under the trace paring $\La\times \La \to \La: (x,y)\mapsto \mathbb{T}_{\La/K}(xy)$, i.e.,
\begin{equation}\label{dual-log}
R_{\La}=\{\,x\in \La\,:\, \mathbb{T}_{\La/K}\big(x\,\l_f(y)\big)\in C,\ \forall y\in F(\mu_{\La})\, \}
\end{equation} 
(cf.  \cite{Florez} Section 3.2.2  for all the properties of $T_{\La}$ and $R_{\La}$). We are now ready to state the main result.

\begin{thm}\label{Lubinazo} 
Take $s\geq \max\{r', n+r+\log_q(e(\La/\Ka_r))\}$. 
Then $\normalfont \text{Tr}_s$ takes 
$(\pi^s/(\pi_1\gamma_s))P_s$ to $\pi^nR_\La$ so 
that it induces a homomorphism
\[
\normalfont \text{Tr}_s: \frac{\ \frac{1}{\gamma_s}P_s\ }{\frac{\pi^s}{\pi_1\gamma_s}P_s} \longrightarrow \La/\pi^nR_\La
\]
and the following formula holds
\begin{equation}\label{phiLubin}
\big(\,N_{s}(\alpha)\,,\,x\,\big)_{\La,n}
=\left[\, \mathbb{T}_{\La/K}\big(\normalfont \, \text{Tr}_s\big(\,QL_s(\alpha)
\, l_f(x)\, \big)\,\right]_f (e_n)
\end{equation}
for all $x\in F(\mu_{\La})$ and  all $\alpha\in K_d(\La_s)'$.

\end{thm}
\begin{proof}
Recalling the notation of Equation \eqref{notation-kummer-pairing},  Equation \eqref{phiLubin}
is equivalent to
\begin{equation}\label{phi-phi}
\big(\,N_{s}(\alpha)\,,\,x\,\big)_{\La,n}^1
= \mathbb{T}_{\La/K}\big(\normalfont \, \text{Tr}_s\big(\,QL_s(\alpha)
\, l_f(x)\, \big) \pmod{\pi^nC}.
\end{equation}

Let $\Ma=\La_t$ and take $\pi_{\Ma}=\gamma_t$. The strategy of the proof is the following: First, we will show that
$QL_t$ coincides with the logarithmic derivative $\mathfrak{Dlog}_{\Ma,m}^1$ from Theorem 5.3.1 of \cite{Florez}; where $t$, $k$ and $m$ are as in the statement of that theorem. 
Then, we will descend to the logarithmic derivative $QL_s$, for 
$s\geq \max\{r', n+r+\log_q(e(\La/\Ka_r))\}$, using the fact that 
$\text{Tr}_s((\pi^s/\pi_1\gamma_s)P_s)\subset \pi^nR_{\La}$; this fact will be demonstrated at the end of the proof.

Let $v$ be the normalized valuation $v_{\Ma}/v_{\Ma}(p)$ and let $R_{t,1}=R_{\Ma,1}$. 
Since $K_1/K$ is totally ramified and $[K_1/K]=q-1$, then $\pi_1^{\varrho(q-1)}\sim \gamma_t^{v_{\Ma}(p)}\sim p$. So $\pi_1$ divides $\gamma_t^{v_{\Ma}(p)/(p-1)}$ and we have 
\[
\frac{1}{\gamma_t}P_t
=\frac{1}{\gamma_t\pi_1D(\La_t/K)}\OZ_{\La_t}
\subset \frac{1}{\gamma_t^{v_{\Ma}(p)/(p-1)+1}D(\La_t/K)}\OZ_{\La_t}
=R_{t,1}.
\]
If $k<t$, then $\pi_1|\pi^{t-k}\gamma_t$, which implies $(\pi^k/\gamma_t)|(\pi^t/\pi_1)$ and so
\[
\frac{\pi^t}{\pi_1\gamma_t}P_t \subset \frac{\pi^k}{\gamma_t}R_{t,1}.
\]
Consider $t$, $k$ and $m$ as in Theorem 5.3.1 of \cite{Florez}. 
In particular, since $t=2k+\varrho+1$ then $k<t$, and 
therefore we can look at the 
factorization of 
$Q_t: \OMa^d\to P_t/(\pi^t/\pi_1)P_t$ 
into $R_{t,1}/(\pi^k/\gamma_t)R_{t,1}$. 
Let $b'$ be as in Lemma \ref{tech-lemma} ( applied to $s=t$), namely, $dT_1\wedge\cdots \wedge w_t=b'\,dT_1\wedge \cdots \wedge \pi_{\Ma}$, which by Equation \eqref{myst-b} can be taken to be of the form $b'=l_f'(e_t)g'(\pi_\Ma)$, where $g(X)$ is a polynomial in $\mathcal{O}_{\Ma_{(0)}}[X]$ such that $g(\pi_\Ma)=e_t$.  
Then for $(a_1,\dots, a_d):=(T_1,\dots,T_{d-1}, e_t)$ we have that 
$$\det \left[ \frac{\partial a_i}{\partial T_j}\right]
=\det  
\begin{pmatrix}
1& 0 &\cdots & \frac{\partial g}{\partial T_1}(\pi_\Ma)\\
0 & 1 &\cdots & \frac{\partial g}{\partial T_2}(\pi_\Ma)\\
\vdots& \vdots & & \vdots \\
0&0 &\cdots & g'(\pi_{\Ma})
\end{pmatrix}
=g'(\pi_{\Ma}).$$ 
Thus from  Equation \eqref{general-derivation-def} we obtain
\[
Q_t(\Td,e_t)=\frac{T_1\cdots T_{d-1}g'(\pi_{\Ma})}{\pi^tb'}=\frac{T_1\cdots T_{d-1}}{\pi^tl_f'(e_t)}.
\]
According to \eqref{inv-form} we have $\overline{c}_{\beta:1}=-1/\pi^t\pmod{\pi^kR_{t,1}}$, where $\overline{c}_{\beta:1}$ is the constant from Equation \eqref{defi-invi-general}. From this we obtain
\begin{equation}\label{invi}
Q_t(\Td,e_t)=\frac{T_1\cdots T_{d-1}}{\pi^tl_f'(e_t)}
=\frac{-T_1\cdots T_{d-1}\overline{c}_{\beta:1}}{l'(e_t)} \pmod{\frac{\pi^k}{\gamma_t}R_{t,1}}.
\end{equation} 
On the other hand, from Equation  \eqref{qmsl} we have
\begin{equation}\label{inv-derii}
Q_t(\Td,e_t)=\frac{T_1\cdots T_{d-1}g'(\pi_{\Ma})}{\pi^tb'}=g'(\pi_{\Ma})Q_t(\Td,\pi_{\Ma}).
\end{equation}

If we now let $\mathfrak{b}:=g'(\pi_{\Ma})=b'/l_f'(e_t)$ then, from Lemma \ref{tech-lemma} and Equation \eqref{new-diff-iden} (applied to $s=t$), we have $dT_1\wedge \cdots \wedge de_t=\mathfrak{b}\,dT_1\wedge \cdots \wedge d\pi_{\Ma}$. Moreover, from Equations \eqref{invi} and \eqref{inv-derii}, we see that 
$$\gamma_i:=Q_t(\Td,\pi_{\Ma})/(T_1\dots T_{d-1})\in R_{t,1}/(\pi^k/\gamma_t)R_{t,1}$$ 
is such that $\mathfrak{b}\gamma_i=-\overline{c}_{\beta:1}/l_f'(e_t)$. 
From  \cite{Florez} Lemma 5.2.1 all such $\gamma_i\in R_{t,1}/(\pi^k/\gamma_t)R_{t,1} $ satisfying the equation $\mathfrak{b}\gamma_i=-\overline{c}_{\beta:1}/l_f'(e_t)$
coincide when reduced to $R_{t,1}/(\pi^m/\gamma_t)R_{t,1}$. 
 In particular, this implies that
$$Q_t(T_1,\dots, T_{d-1},\pi_{\Ma})=\mathfrak{D}_{\Ma,m}^1(T_1,\dots, T_{d-1}, \pi_{\Ma})\pmod{(\pi^m/\gamma_t)R_{t,1}},$$ 
where $\mathfrak{D}_{\Ma,m}^1$ is the derivation from Definition 5.2.2 of \cite{Florez}. 
Therefore,  by \cite{Florez} Proposition 4.2.3, this guarantees that the reduction of $Q_t$ to $R_{t,1}/(\pi^m/\gamma_t)R_{t,1}$ 
coincides with  $\mathfrak{D}_{\Ma,m}^1$: Indeed, both maps are $d$-dimensional derivations that coincide at $(T_1,\dots,\pi_{\Ma})$. As a consequence of this, the map $\mathfrak{Dlog}_{\Ma,m}^1$
from \cite{Florez} Definition 5.2.2. coincides with the map
$$
QL_t:K_d(\Ma)\longrightarrow \frac{\frac{1}{\gamma_t}P_t}{\frac{\pi^t}{\pi_1\gamma_t}P_t}$$
 from Definition \ref{defilQL}. Therefore, we can replace  $\mathfrak{Dlog}_{\Ma,m}^1$ by  $QL_t$ in \cite{Florez} Theorem 5.3.1 to obtain 
 \begin{equation}\label{key-id-res}
(N_{\La_t/\La}(\epsilon),x)_{\La,n}^1=\mathbb{T}_{\La_t/K}(QL_t(\epsilon)\,l_f(x)) \pmod{\pi^nC}, 
  \end{equation}
for all $\epsilon \in K_d(\La_t)$ and all $x\in F(\mu_{\La})$.

On the other hand, by Proposition \ref{levelQL}  we know that for $ s\geq r'$ the following identity holds for all $\epsilon\in K_d(\La_t)$
\begin{equation}\label{trace-QL-t-s}
QL_s\big(\,N_{t/s}(\epsilon)\,\big)= \normalfont \text{Tr}_{t/s}\big(\,QL_t(\epsilon)\,\big) \pmod{\frac{\pi^s}{\pi_1\gamma_s}P_s}. 
\end{equation}
Suppose for the moment that 
\begin{equation}\label{condition-trace}
\text{Tr}_s\big(\,(\pi^s/(\pi_1\gamma_s))P_s\,\big)\subset \pi^nR_{\La}\ 
\mathrm{for}\ s\geq n+r + \log_q(e(\La/\Ka_r)). 
\end{equation}
Then taking $\text{Tr}_s$ in \eqref{trace-QL-t-s} we get
\begin{align*}
\text{Tr}_t\big(\, QL_t(\epsilon)\,\big)
&=\text{Tr}_s\big(\, \text{Tr}_{t/s}\big(\, QL_t(\epsilon)\, \big)\, \big)\\
&=\text{Tr}_s\left(\, QL_s\big(\,N_{t/s}(\epsilon)\,\big)\, \right) \pmod{\pi^nR_{\La}}.
\end{align*}
Thus, multiplying by $l_f(x)$, for $x\in F(\mu_{\La})$, and taking into account the definition of $R_{\La}$ 
(cf. Equation \eqref{dual-log}), we obtain for all $\epsilon\in K_d(\La_t)$ that
\begin{equation}\label{result-of-trac}
\mathbb{T}_{\La/K}\big(\,\text{Tr}_t\big(\, QL_t(\epsilon)\,\big)l_f(x)\,\big)
=\mathbb{T}_{\La/K}\big(\,\text{Tr}_s\left(\, QL_s\big(\,N_{t/s}(\epsilon)\,\big)\, \right)l_f(x)\,\big) \pmod{\pi^nC}.
\end{equation}
Combining \eqref{key-id-res} and \eqref{result-of-trac} we get 
 \[
(N_{\La_t/\La}(\epsilon),x)_{\La,n}^1=\mathbb{T}_{\La_s/K}(QL_s(N_{t/s}(\epsilon))l_f(x)) \pmod{\pi^nC}. 
 \]
 for all $\epsilon \in K_d(\La_t)$ and all $x\in F(\mu_{\La})$.

To finally  prove \eqref{phi-phi}  from the above identity, we now take an $\alpha\in K_d(\La_s)'$. In particular, from the definition of $K_d(\La_s)'$, we have  that $\alpha\in N_{t/s}(K_d(\La_t))$ 
for the aforementioned $t$, i.e., $t=2k+\varrho+1$ and $k$ sufficiently large; exactly as required in \cite{Florez} Theorem 5.3.1. Thus $\alpha=N_{t/s}(\epsilon)$ for some $\epsilon \in K_d(\La_t)$,
and hence we conclude that
\[
(N_{s}(\alpha),x)_{\La,n}^1=\mathbb{T}_{\La_s/K}(QL_s(\alpha)l_f(x)) \pmod{\pi^nC}, 
 \]
for $s\geq \max\{r', n+r+\log_q(e(\La/\Ka_r))\}$, for all $\alpha\in K_d(\La_{s})'$ and all $x\in F(\mu_{\La})$. This is exactly Equation \eqref{phi-phi}.

Thus, it  remains only to prove \eqref{condition-trace}.
To do this, let us take $x\in F(\mu_\La)$. From $f(x)\equiv x^{q} \pmod{\pi x}$ we obtain $f^{(m)}(x)\equiv x^{q^m} \pmod{\pi x}$  for all $m\geq 1$; recall that $f^{(m)}$ is the power series defined at the beginning of Section  \ref{Description of the results}. Therefore
\begin{equation}\label{valua}
v(f^{(m)}(x))\geq \min \{v(x^{q^m}),\ v(\pi x)\} \quad (\,\forall m\geq 1\,).
\end{equation}

Let 
\[
s'\geq \log_q\left( \frac{e(\La/\Ka)}{q-1} \right)+1=r+\log_q\left( e(\La/\Ka_r) \right),
\]
which is equivalent to
\begin{equation}\label{ineqqq}
\frac{q^{s'-1}}{e(\La/K)}\geq \frac{1}{q-1}.
\end{equation}
We  have from \eqref{valua} that $v(f^{(s'-1)}(x))\geq \min \{v(x^{q^{s'-1}}),\ v(\pi x)\}$.
On the other hand, the  Inequality \eqref{ineqqq} implies
\[v(x^{q^{s'-1}})=q^{s'-1}v(x)
=q^{s'-1}\,\frac{v_{\La}(x)}{\,e(\La/\Ka)\,\varrho\,}
\geq \frac{v_{\La}(x)}{(q-1)\varrho}\geq \frac{1}{(q-1)\varrho}=\frac{v(\pi)}{q-1},
\]
which combined with the obvious inequality 
 \[
 v(\pi x)> v(\pi)>\frac{v(\pi)}{q-1},
\]
yields $v(f^{(s'-1)}(x))\geq v(\pi)/(q-1)$; here we are using the fact that $v(\pi)=1/\varrho$. Applying again Equation (\ref{valua})  to $x=f^{(s'-1)}(x)$ ( and $m=1$) we obtain
\begin{align*}
v(f^{(s')}(x))
&\geq \min\ \left\lbrace v\left( \ ( f^{(s'-1)}(x))^{q}\ \right) ,\ v(\pi f^{(s'-1)}(x))\ \right\rbrace \\
&\geq \left( 1+\frac{1}{q-1}\right) v(\pi)> \frac{v(\pi)}{q-1}.
\end{align*}
Thus, by Lemma \ref{local-iso-log} below,   we have  
$v(l_f(f^{(s')}(x)))=v(f^{(s')}(x))$. Moreover, since  $v(\pi)/(q-1)=v(\pi_1)$ and $l_f\circ f^{(m)}=\pi^m\, l_f$ for all $m\geq 1$, then
$$v(\pi^{s'}l_f(x))=v(f^{(s')}(x))\geq \left(1+\frac{1}{q-1}\right)v(\pi)=v(\pi)+v(\pi_1).$$ 
This implies $(\pi^{s'-1}/\pi_1)T_{\La}\subset \mathcal{O}_{\La}$, bearing in mind that $T_{\La}=l_f(F(\mu_\La))$. Taking duals with respect to $\mathbb{T}_{\La/K}$, and recalling that $R_{\La}$ is the dual of $T_{\La}$ (cf. Equation \eqref{dual-log}), we get
\[
\frac{\pi^{s'-1}}{\pi_1}\frac{1}{D(\La/K)}\mathcal{O}_{\La}\subset R_{\La}.
\]
If we now take $s\geq s'$ then, keeping in mind that  $\text{Tr}_s(P_s)\subset P_{\La}$ (cf.  Equation (\ref{TrazaP_t})),  we see from the above containment  that
$\text{Tr}_s$ takes 
\[
\frac{\pi^s}{\pi_1\gamma_s}P_{s}\subset \frac{\pi^s}{\pi_1\pi_\La}P_{s}
\]
to 
\[
 \frac{\pi^s}{\pi_1\pi_\La}P_{\La}=\frac{\pi^s}{\pi_1^2\pi_\La}\frac{1}{D(\La/K)}\mathcal{O}_{\La}\subset \frac{\pi^{s-s'+1}}{\pi_1\pi_\La} R_{\La}.
\]
Finally, observing  that $q\neq 2$ ( since $p\neq 2$) implies $\pi_1\pi_\La|\pi$, because $$v_{\La}(\pi/(\pi_1\pi_\La))=e(\La/K)-e(\La/K_1)-1=e(\La/K_1)(e(K_1/K)-1)-1=e(\La/K_1)(q-2)-1$$ 
and $e(K_1/K)=q-1$, we conclude that if $s\geq n+s'$ then $\text{Tr}_s$ takes $(\pi^s/(\pi_1\gamma_s))P_s$ to $\pi^nR_{\La}$. This proves \eqref{condition-trace} and hence the theorem in its entirety.
\end{proof}

\begin{lemma}\label{local-iso-log}
The formal logarithm $l_f$ is a local  isomorphism for the subgroup of $F(\mu_L)$ 
of elements $w$ such that
\[
v(w)>1/(\varrho(q-1))=v(\pi)/(q-1),
\]
and furthermore for these elements we have that
\[
v(l_f(w))=v(w).
\]
\end{lemma}
\begin{proof}
The inequality $v(w)>v(\pi)/(q-1)$ is equivalent to 
$v(w^q)>v(\pi)+v(w)$, and since 
$f(x)\equiv \pi x+x^q \pmod{\pi x^2}$ this implies  
$v(f(w))=v(\pi w)$. Thus $v(f^{(n)}(w))=v(\pi^n w)$ 
and therefore we can take $n$ large enough 
such that $v(f^{(n)}(w))>1/(p-1)$. By 
Proposition 6.2.3 of \cite{Florez} or \cite{silver} IV Theorem 6.4 we have 
that $v(l_f(f^{(n)}(w)))=v(f^{(n)}(w))$. Thus $v(\pi^nl_f(w))=v(\pi^n\, w)$; 
recalling that $l_f\circ f^{(n)}=\pi^nl_f$. This proves the lemma.
\end{proof}

\section{Generalized Artin-Hasse and Iwasawa formulas}\label{Comparison of formulas}
Let $\pi$ be a uniformizer of $K$ and let $\varrho$ is the ramification of index of $K/\Qp$. Fix an $f\in \Lambda_{\pi}$ and let $F_{f}$ denote the Lubin-Tate formal group associated to $f$. We will use the following notation in the rest of this article. Let $\Ka=K\TT$ and set  $\Ka_{\pi,n}=K_{\pi,n}\TT$
for $n\geq 1$. For $\La=\Ka_{\pi,n}$ we will use the notation \eqref{notational}. The main results in this section are Corollary \ref{Lemma-Ev-torison-3} and Theorem \ref{Iwasawa-General} which are a refinement of Theorem \ref{Lubinazo} for the field $\Ka_{\pi,n}$.  Corollary \ref{Lemma-Ev-torison-3} is a generalization to Lubin-Tate formal groups of the generalized Artin-Hasse formula  of Zinoviev (cf. \cite{Zinoviev} Corollary 2.1). On the other hand, Theorem \ref{Iwasawa-General} is a generalization  to Lubin-Tate formal groups of the explicit reciprocity laws  of Zinoviev (cf. \cite{Zinoviev} Theorem 2.2 ) and  Kurihara (cf. \cite{Kurihara} Theorem 4.4) for the generalized Hilbert symbol for the field $\Ka_{\pi,n}$.

\subsection{A refinement of Theorem \ref{Lubinazo}}
We start with the following refinement of Theorem \ref{Lubinazo} which will be used later in the deduction of Theorem \ref{Iwasawa-General}.

\begin{thm}\label{Partial-Iwasawa-formulas}
Let $\La=\Ka_{\pi,n}$. The following identity holds:
\begin{equation}\label{eq-Partial-Iwasawa-formulas}
(\alpha, x)_{\La,n}= \left[\,\Tr  \big(\,QL_n(\alpha)\,l_f(x)\,\big)\,\right]_f(e_{f,n})
\end{equation}
for all $\alpha \in K_d(\La)'$ and 
all $x\in F(\mu_{\La})$ such that $v_{\La}(x)\geq 2\,v_{\La}(p)/(\varrho\,(q-1))+1$.
\end{thm}
\begin{proof}
First of all, notice that in this case $n=r=r'$; for $r$ and $r'$ as in Section \ref{The canonical logarithmic derivatives}. Borrowing the notation from Theorem  \ref{Lubinazo}, we have first from
Equation \eqref{key-id-res}  that for all $\epsilon\in K_d(\La_t)$ and all $x\in F(\mu_{\La})$
\begin{equation}\label{nece}
(N_{\La_t/\La}(\epsilon),x)_{\La,n}^1=\mathbb{T}_{\La_t/K}(QL_t(\epsilon)l_f(x)) \pmod{\pi^nC}, 
\end{equation}
 and from \eqref{trace-QL-t-s} (applied to $s=r'=n$) that
\begin{equation}\label{trazalll}
QL_n\big(\,N_{t/n}(\epsilon)\,\big)= \normalfont \text{Tr}_{\La_t/\La}\big(\,QL_t(\epsilon)\,\big) \pmod{\frac{\pi^n}{\pi_1\gamma_n}P_n};
\end{equation}
taking into account that $\text{Tr}_{t/n}=\text{Tr}_{\La_t/\La} $ since $\La_n=\La$. 

Suppose for the moment that
\begin{equation}\label{duali}
v_{\La}(x)\geq  2v_{\La}(p)/(\varrho(q-1))+1
\quad \Longrightarrow \quad 
\Tr\left(\,\left(\frac{\pi^n}{\pi_1\gamma_n}P_n\right)\, l_f(x)\,\right)\subset \pi^n C.
\end{equation}
Then  multiplying both sides of \eqref{trazalll} by $l_f(x)$, for $x\in F(\mu_{\La})$ such that 
$v_{\La}(x)\geq  2v_{\La}(p)/(\varrho(q-1))+1$, and  
then  applying $\Tr$ we obtain from \eqref{duali} that
\[
\Tr\big(\,QL_n\big(\,N_{t/n}(\epsilon)\,\big)\cdot l_f(x)\,\big)= \mathbb{T}_{\La_t/K}\big(\,QL_t(\epsilon)\cdot l_f(x)\,\big) \pmod{\pi^nC}.
\]
Therefore this, combined with \eqref{nece}, implies, for all   $x\in F(\mu_{\La})$ such that 
$v_{\La}(x)\geq  2v_{\La}(p)/(\varrho(q-1))+1$, that
\begin{equation}\label{clave-id}
(N_{\La_t/\La}(\epsilon),x)_{\La,n}^1=\Tr\big(\,QL_n\big(\,N_{t/n}(\epsilon)\,\big)\cdot l_f(x)\,\big) \pmod{\pi^nC}.
\end{equation}
 To finally prove the identity, take $\alpha \in K_d(\La)'=\cap_{m\geq n}K_d(\La_m)$.
 In particular, there exists an $\epsilon\in K_d(\La_t)$ such that $\alpha=N_{\La_t/\La}(\epsilon)$; recalling that  $t$ is as in Theorem \ref{Lubinazo}. Therefore applying the identity \eqref{clave-id} to $\alpha=N_{\La_t/\La}(\epsilon)$ we obtain 
\[
(\alpha, x)_{\La,n}^1=\Tr  \big(\,QL_n(\alpha)\,l_f(x)\,\big)\pmod{\pi^nC},
\]
for all $\alpha\in K_d(\La)'$ and all $x\in F(\mu_{\La})$ such that 
$v_{\La}(x)\geq  2v_{\La}(p)/(\varrho(q-1))+1$. Bearing in mind the notation from \eqref{notation-kummer-pairing}, this yields \eqref{eq-Partial-Iwasawa-formulas}.
 
It remains only to show \eqref{duali}. This follows by noticing first that $(\pi_1^2\gamma_n)\OLa$ if and only if  $v_{\La}(x)\geq  2v_{\La}(p)/(\varrho(q-1))+1$. Second, by noticing that
 $(\pi_1^2\gamma_n)\OLa$
is the dual of $$\frac{\pi^n}{\pi_1\gamma_n}P_{\La}=\frac{\pi^n}{\pi_1^2\gamma_n}\cdot\frac{1}{D(\La/K)}$$ with respect to the pairing $(y,w)\to \Tr(yw)$ mod $\pi^nC$, and finally by observing that 
$x\in (\pi_1^2\gamma_n)\OLa$ if and only if $l_f(x)\in (\pi_1^2\gamma_n)\OLa$,  by
 Lemma \ref{local-iso-log}.
 
\end{proof}

Our main goal is to extend the above theorem to all $\alpha\in K_d(\La)$.  Here we are going to follow Kolyvagin's ideas in \cite{koly} Section 7.4. The crux of Kolyvagin's argument is to vary the uniformizer of $K$ by taking $\xi=u\pi$, for $u\equiv1 \pmod{\pi^n}$, and study how the formulas are transformed when we consider Lubin-Tate formal groups associated to a power series $g\in \Lambda_{\xi}$. Additionally, we need to modify the logarithmic derivative $QL_n$ ( cf. Definition \ref{lift}) so that it is defined  $\text{mod}\,(\pi^n/\pi_1)P_n$ rather than $\text{mod}\,(\pi^n/(\pi_1\,\gamma_n)P_n$. This will allow us to deal with the restriction on $x\in F(\mu)$ in Theorem \ref{Partial-Iwasawa-formulas} from $v_{\La}(x)\geq 2\,v_{\La}(p)/(\varrho\,(q-1))+1$ to $v_{\La}(x)\geq 2\,v_{\La}(p)/(\varrho\,(q-1))$ (cf. Corollary  \ref{Iwa-Id-Wiles group}).

\subsection{Varying the uniformizer of $K$}\label{Varying the uniformizer of $K$}

Let $\xi$ be a uniformizer of $K$.  For $g\in \Lambda_{\xi}$ we let $F_g$ denote the formal group associated to $g$. We will add a $g$ to the subscript of all the elements associated to the formal group $F_g$. In this way,
 the corresponding formal logarithm, torsion points, $n$th torsion group, Kummer pairing, logarithmic derivatives, etc., will be denoted by $l_g$, $e_g$, $e_{g,n}$, $\kappa_{g,n}$, $(,)_{g,\La,n}$, $QL_{g,s}$, etc., respectively.  Recall, from  analogous  considerations made in Section \ref{Description of the results}, that we are assuming  $e_g$ is a generator of $\kappa_{g}:=\varprojlim \kappa_{g,n}$ ($\simeq C$) as $C$-modules and $e_{g,n}$ is the restriction of $e_g$ to $\kappa_{g,n}$. This implies, in particular, that $g(e_{g,m+1})=e_{g,m}$ for all $m\geq 1$.
 Since we already fixed the power series $f\in \Lambda_{\pi}$, sometimes we will omit the subscript $f$ in this case.

 In this subsection we will collect some notation and 
 results that will inform us about the  passage 
 from the formal group $F_f$ to the formal group $F_g$.

Let $K_{\xi,n}=K(\kappa_{g,n})$. We start with the following result of Kolyvagin (cf. \cite{koly} Proposition 7.25).

\begin{prop}\label{importa}
$K_{\xi,n}=K_{\pi,n}$ if and only if $u\equiv 1 \pmod{ \pi^n}$, where $u=\xi/\pi$.
\end{prop}

Now let $R$ be a ring containing $C=\mathcal{O}_K$. Denote by $\text{Hom}_R(F_f,F_g)$ be the ring of power series $t(X) \in XR[[X]]$ such that $$t(F_f(X,Y))=F_g(t(X),t(Y)).$$
According to \cite{koly} Proposition 1.1 we have an injective ring homomorphism $c: \text{Hom}_R(F_f,F_g)\to R$ defined by $t(X)=t_1X+\cdots\mapsto c(t):=t_1$. For $a\in c(\text{Hom}_R(F_f,F_g))$
we denote by $[a]_{f,g}(X)$ the power series such that $c([a]_{f,g})=a$.

Let $u$ be a principal unit of $K$. According to \cite{koly} Proposition 7.1 and Corollary 7.5 there exists a unit $\epsilon\equiv 1 \pmod{(u-1)}$  in  $\hat{C}_{nr}$, the ring of integers of the completion of the maximal unramified extension  $K^{nr}$ of $K$, such that  
\[
\text{Fr}_K(\epsilon)=\epsilon u \quad \text{and} \quad [\epsilon]_{f,g}\in \text{Hom}_{\hat{C}_{nr}}(F_f,F_g).
\] 
Therefore the power series $[\epsilon]_{f,g}(X)$
 is an isomorphisms between $F_f$ and $F_g$ over $\hat{C}_{nr}$ and
 \begin{equation*}
[\epsilon]_{f,g}^{ \text{Fr}_{K}}=[\epsilon\,u]_{f,g}.
 \end{equation*}

Suppose further that the unit $u$ is such that $u\equiv 1 \pmod{\pi^n}$. Let $L=K_{\pi,n}$. Observing that $\text{Fr}_{L}=\text{Fr}_K$ is the Frobenius automorphism of $L^{nr}$ over $L$, then $\text{Fr}_{L}(\epsilon)= \epsilon u \equiv \epsilon \pmod{\pi^n}$.

Let $\La=\Ka_{\pi,n}$. Using the notation from the proof of Lemma \ref{one-two-connec} we define, according to Fesenko c.f \cite{Zhukov} Section 6.4.1., the \textit{valuation} map $\mathfrak{v}_{\La}$  by
\[
\mathfrak{v}_{\La}:\,K_d(\La)\ \xrightarrow{\ \partial\ }\ K_{d-1}(\La_{(d-1)})\ \xrightarrow{\ \partial \ } \cdots \xrightarrow{\ \partial\ }\  K_0(\La_{(0)})=\Z.
\]
Observe, for example, that $\mathfrak{v}_{\La}(\{T_1,\dots, T_{d-1},\pi_{\La}\})=1$. For $\alpha\in K_d(\La)$ the Galois element  $\Upsilon_{\La}(\alpha)\in$\text{ Gal}$(\La^{ab}/\La)$ can be restricted to Gal$(L^{nr}/L)$. According to the commutative diagram from the proof of Lemma \ref{one-two-connec} this restriction coincides with $\text{Fr}_{L}^{\mathfrak{v}_{\La}(\alpha)}$, thus
\begin{equation}\label{Frob-higher}
\Upsilon_{\La}(\alpha)(\epsilon)=\epsilon\,u^{\mathfrak{v}_{\La}(\alpha)}\quad \text{and} \quad  [\epsilon]_{f,g}^{\Upsilon_{\La}(\alpha)}=[\epsilon\,u^{\mathfrak{v}_{\La}(\alpha)}]_{f,g}.
\end{equation}

\subsection{Generalized Artin-Hasse formulas}
In this section we will prove a couple of results that can be regarded as a generalization of the
 Artin-Hasse  formula (cf. Equation  \eqref{Hasselin}) for the Kummer pairing associated to a higher local field. In particular, we will deduce the formulas of Zinoviev  \cite{Zinoviev} Corollary 2.1.

Denote by $\Pi\Lambda_{\xi}$  the subset of $\Lambda_{\xi}$  consisting of polynomials of degree $q$  having leading coefficient 1.  The Lubing-Tate formal group associated to a $g\in \Pi\Lambda_{\xi}$ have further interesting properties as can be stated in the following result of Kolyvagin.

\begin{prop}
Let $g\in \Pi\Lambda_{\xi}$. For $m>0$ and $p$ an odd prime we have 
\begin{equation}\label{norm-torsion}
\prod_{v\in \kappa_{g,m}}F_g(X,v)=(-1)^{p-1}g^{(m)}(X)=g^{(m)}(X)
\end{equation}
and 
\begin{equation}\label{trace-torsion}
\sum_{v\in \kappa_{g,m}} 
      \frac{1}{\xi^{n+m}\,l'_g(F_g(e_{g,n+m},v))\,F_g(e_{g,n+m},v)}
   =
   \frac{1}{\xi^n\,l'_g(e_{g,n})\,e_{g,n}}\, .
\end{equation}
Moreover, let $\La_m=K_{\xi,m}\TT$. Then for $t\geq s\geq 1$ we have
\begin{equation}\label{real-norm-torsion}
N_{\La_t/\La_s}(e_{g,t})=e_{g,s}
\end{equation}
and 
\begin{equation}\label{real-trace-torsion}
\normalfont \text{Tr}_{\La_t/\La_s} \left(\frac{1}{\,\xi^t\,l'_g(e_{g,t})\,e_{g,t}\,} \right)=\frac{1}{\,\xi^s\,l'_g(e_{g,s})\,e_{g,s}\,}\, .
\end{equation}
\end{prop}
\begin{proof}
For Equation \eqref{norm-torsion} see \cite{koly} Proposition 7.21, and for Equation \eqref{trace-torsion} see Equation TR1: 7.4.

We now prove the remaining part of the proposition. 
First, recall from Subsection \ref{Varying the uniformizer of $K$}  that the torsion points $e_{g,m}\in \kappa_{g,n}$ were chosen to be compatible, i.e., $g(e_{g,m+1})=e_{g,m}$ for $m\geq 1$. Now let $m=t-s$. We claim that as $\sigma$ 
runs through all the elements of $\text{Gal}(\La_t/\La_s)$, then $\sigma(e_{g,t})\ominus_{g}e_{g,t}$ runs through all the elements of $\kappa_{g,m}$; where $\ominus_g$ denotes subtraction in the formal group $F_g$.
Indeed, this is true since $\La_t=\La_s(e_{g,t})$, and $[\La_t:\La_s]$ and $|\kappa_{g,m}|$ are both equal to $q^m$, as it is stated in the introduction of Section \ref{Canonical derivations and logarithmic derivatives}. Therefore by \eqref{norm-torsion}
\[
N_{\La_t/\La_s}(e_{g,t})
=\prod_{\sigma\in \text{Gal}(\La_t/\La_s)} \sigma(e_{g,t})
=\prod_{v\in \kappa_{g,m}}F_g(e_{g,t},v)=g^{(m)}(e_{g,t})=e_{g,s},
\]
and by \eqref{trace-torsion}
\begin{align*}
\normalfont \text{Tr}_{\La_t/\La_s} \left(\frac{1}{\,\xi^t\,l'_g(e_{g,t})\,e_{g,t}\,} \right)
&=\sum_{\sigma\in \text{Gal}(\La_t/\La_s)}\frac{1}{\,\xi^t\,l'_g(\sigma(e_{g,t}))\,\sigma(e_{g,t})\,}\\
&=\sum_{v\in \kappa_{g,m}}\frac{1}{\,\xi^t\,l'_g(F_g(e_{g,t},v))\,F_g((e_{g,t},v)\,} \\
&=\frac{1}{\,\xi^s\,l'_g(e_{g,s})\,e_{g,s}\,}.
\end{align*}
\end{proof}

These useful identities will give us more precise results for Kummer pairings of  Lubin-Tate formal groups associated to polynomials $g\in \Pi\Lambda_{\xi}$. Indeed, we have the following result. 
\begin{prop}\label{Lemma-Ev-torison}
Let $g\in \Pi\Lambda_{\xi}$ for $\xi$ a uniformizer of $K$. Let $e_{g,n}$ be a generator of $\kappa_{g,n}$ and let $\La= K_{\xi,n}\TT$. Then the following identity holds
\[\,
\big(\,\{u_1,\dots, u_{d-1},e_{g,n}\}\,,\,x\,\big)_{g,\La,n}
=\left[ 
\Tr\left(\,\frac{\det \left[\frac{\partial u_i}{\partial T_j}\right]}{u_1\cdots u_{d-1}}\,\frac{\Tdc}{\,\xi^n\,l'_g(e_{g,n})\,e_{g,n}\,}\, l_g(x)\,\right)\,
\right]_g (e_{g,n})
\]
for all units $u_1,\dots, u_{d-1}$ of  $\La$ and  all $x\in F_g(\mu_{\La})$.
\end{prop}
\begin{proof}
Let $\La_s=\La(\kappa_{g,s})$  for $s\geq n$. First,  we have from \eqref{real-norm-torsion} that $N_{\La_t/\La_s}(e_{g,t})=e_{g,s}$ for all $t\geq s\geq n$. Therefore 
\[
N_{\La_t/\La_s}(\{u_1,\dots,e_{g,t}\})=\{u_1,\dots,N_{\La_t/\La_s}(e_{g,t})\}=\{u_1,\dots,e_{g,s}\}.
\] 
This implies, in particular, that $N_{\La_s/\La}(\{u_1,\dots,e_{g,s}\})=\{u_1,\dots,e_{g,n}\}$
and also that $\{u_1,\dots,e_{g,s}\}$ belongs to $K_d(\La_s)':=\cap_{t\geq s}N_{\La_t/\La_s}( K_d(\La_s))$. On the other hand, by the definition of $QL_{g,s}$ we have
\[
QL_{g,s}(u_1,\dots, u_{d-1},e_{g,s})=\frac{\det \left[\frac{\partial u_i}{\partial T_j}\right]}{u_1\cdots u_{d-1}}\,\frac{\Tdc}{\,\xi^s\,l'_g(e_{g,s})\,e_{g,s}\,}\ .
\]
Moreover, from \eqref{real-trace-torsion} it follows that
\[
\text{Tr}_{\La_s/\La} \left(\frac{1}{\,\xi^s\,l'_g(e_{g,s})\,e_{g,s}\,} \right)=\frac{1}{\,\xi^n\,l'_g(e_{g,n})\,e_{g,n}\,}\, .
\]

From these observations the result follows now from Theorem \ref{Lubinazo} applied  to the formal group $F_g$
and the element $\alpha=\{u_1,\dots,e_{g,s}\}$, by taking $s$ large enough (as required in the statement of that theorem).
\end{proof}

If we take $\xi=\pi$ in the above proposition and notice that, by  \cite{koly} Corollary 7.5, the power series $t(X)=[1]_{f,g}(X)$
is an isomorphisms ( over $C$) between the formal groups $F_f$ and $F_g$, we obtain the following result.
\begin{prop}\label{Lemma-Ev-torison-2}
Let $\La =\Ka_{\pi,n}$. Let $g\in \Pi\Lambda_{\pi}$ and $e_{g,n}=[1]_{f,g}(e_{f,n})$. Then the following identity holds
\[\,
\big(\,\{u_1,\dots, u_{d-1},e_{g,n}\}\,,\,x\,\big)_{\La,n}
=\left[ 
\Tr\left(\,\frac{\det \left[\frac{\partial u_i}{\partial T_j}\right]}{u_1\cdots u_{d-1}}\,\frac{\Tdc}{\,\xi^n\,l'_g(e_{g,n})\,e_{g,n}\,}\, l_f(x)\,\right)\,
\right]_f \,(e_{f,n})
\]
for all units $u_1,\dots, u_{d-1}$ of  $\La$ and  all $x\in F_f(\mu_{\La})$.
\end{prop}
\begin{proof}
Since $t(X)=[1]_{f,g}(X)$ is an isomorphisms (over $C$) between the formal groups $F_f$ and $F_g$ then, by
 \cite{Florez} Proposition 2.2.1.(8), we have  that $$t\big(\,(\alpha,x)_{\La,n}\,\big)=(\alpha,t(x))_{g,\La,n}.$$ The result follows now from Proposition \ref{Lemma-Ev-torison}; keeping Proposition \ref{importa} in mind.
\end{proof}

When we specialize the above proposition to the units $u_1=T_1,\dots, u_{d-1}=T_{d-1}$ we obtain the following 
generalization of the  Artin-Hasse formula \cite{Hasse} to arbitrary higher local fields. 

\begin{coro}\label{Lemma-Ev-torison-3}
 Let $\Ma$ be an arbitrary higher local field with a system of local uniformizers $T_1,\dots, T_{d-1}$ and $\pi_{\Ma}$ such that $\Ma \supset \Ka_{\pi,n}$. Let  $g\in \Pi\Lambda_{\pi}$ and $e_{g,n}=[1]_{f,g}(e_{f,n})$. Then the following identity holds
\[\,
\big(\,\{T_1,\dots, T_{d-1},e_{g,n}\}\,,\,x\,\big)_{\Ma,n}
=\left[ 
\mathbb{T}_{\Ma/S}\left(\,\frac{1}{\,\xi^n\,l'_g(e_{g,n})\,e_{g,n}\,}\, l_f(x)\,\right)\,
\right]_f \,(e_{f,n})
\]
for   all $x\in F_f(\mu_{\Ma})$.
\end{coro}
\begin{proof}
Let $\La=\Ka_{\pi,n}$. Since $\Ma/\La$ is a finite extension we know,
 by  \cite{Florez} Proposition 2.2.1. (7), that  
$(\alpha,x)_{\Ma,n}=(\alpha,N_{\Ma/\La}^{F_f}(x))_{\La,n}$ for all $\alpha\in K_d(\La)$ and all $x\in F_f(\mu_{\Ma})$; here $N_{\Ma/\La}^{F_f}(x)=\oplus_{F_f,\,\sigma}\,x^{\sigma}$ as 
         $\sigma$ ranges over all embeddings  of $\Ma$ in $\overline{\La}$ over $\La$. Therefore it is enough to show the result for $\La= \Ka_{\pi,n}$, but this is a consequence of Proposition \ref{Lemma-Ev-torison-2}. 
\end{proof}

As we remarked above, Corollary \ref{Lemma-Ev-torison-3} is also a generalization of  \cite{Zinoviev} Corollary 2.1 (25) to Lubin-Tate formal groups. Indeed, simply let $\zeta_{p^n}$ be a primitive $p^n$th root of 1 and  take $K=\Qp$, $\pi=p$, $f(X)=g(X)=(X+1)^p-1$,  $l_f(X)=\log(X+1)$, $\La=\Qp(\zeta_{p^n})\TT$, $e_{f,n}=\zeta_{p^n}-1$ and $F_f=F_{g}$ to be the multiplicative group $F_m(X,Y)=X+Y+XY$. On the other hand, the formula \cite{Zinoviev} Corollary 2.1 (24), namely,
\[
\big(\,\{T_1,\dots, T_{d-1},\zeta_{p^n}\}\,,\,x\,\big)_{\La,n}
=\left[ 
\mathbb{T}_{\La/\Qp}\left(\,\frac{1}{\,p^n\,}\, l_f(x)\,\right)\,
\right]_f \,(e_{f,n})\, , \quad (\,x\in F_m(\mu_{\La})\,)
\]
is obtained in a similar fashion from Theorem \ref{Lubinazo}.  This time noticing first that, for 
$\La_s=\Qp(\zeta_{p^s})\TT$ ($s\geq 1$) with $\{\zeta_{p^s}\}$ a collection of primitive $p^s$th roots of 1 such that $\zeta_{p^{s+1}}^p=\zeta_{p^s}$, we have $N_{\La_s/\La}(\zeta_{p^s})=\zeta_{p^n}$ 
for all $s\geq n$, from which follows that $\{T_1,\dots, T_{d-1},\zeta_{p^n}\}\in K_d(\La)'$,
and also by noticing that
\[
QL_s(T_1,\dots, T_{d-1},\zeta_{p^s})=\frac{1}{p^s}\,.
\]

Furthermore, the following stronger result is true:
\begin{equation}\label{Stronger-Artin-Hasse}
\big(\,\{u_1,\dots, u_{d-1},\zeta_{p^n}\}\,,\,x\,\big)_{\La,n}
=\left[ 
\mathbb{T}_{\La/\Qp}\left(\,\frac{1}{\,p^n\,}\, \frac{\det \left[\frac{\partial u_i}{\partial T_j}\right]}{u_1\cdots u_{d-1}} \,l_f(x)\,\right)\,
\right]_f (e_{f,n})
\end{equation}
for all units $u_1,\dots, u_{d-1}$ of  $\La$ and  all $x\in F_m(\mu_{\La})$. This result, as well as Proposition \ref{Lemma-Ev-torison-2}, is not covered  in any of the reciprocity laws in the literature.
\subsection{The lift of $QL_n$}\label{The lift of $QL_n$}

In order to  extend Theorem \ref{Partial-Iwasawa-formulas} to all $\alpha \in K_d(\La)$ we need to modify the logarithmic derivative $QL_n$ so that it is defined $\pmod{(\pi^n/\pi_1)P_n}$. This is the content of the following definition. 

\begin{defin}\label{lift}
Let $\Ma=\Ka_{\pi,s}$ for $s\geq n$. We define a lift of  $QL_s$ to a homomorphism
\[
QL_s: K_d(\Ma)\ \longrightarrow \ \frac{\ \frac{1}{\gamma_s}P_s\ }{\ \frac{\pi^s}{\pi_1}P_s\ }\ .
\]
as follows. First, for elements $\{u_1,\dots,u_d\}\in K_d(\Ma)$, with  
units $u_1,\dots, u_d$ in $\Ma$, the $QL_s$ is already defined mod $(\pi^s/\pi_1)\,P_s$. We set now 
\[
QL_s(u_1,\dots, u_{d-1},e_{f,s})=\frac{\det \left[\frac{\partial u_i}{\partial T_j}\right]}{u_1\cdots u_{d-1}}\,\frac{\Tdc}{\pi^s\,l_f'(e_{f,s})\,e_{f,s}}\, ,
\]
where $\frac{\partial a}{\partial T_j}$, $j=1,\dots, d-1$, are the canonical derivations from Section \ref{Terminology and Notation}. Next we define
\[
QL_s(u_1,\dots, u_d\cdot (e_{f,s})^k)=QL_s(u_1,\dots, u_{d})+k\,QL_s(u_1,\dots, e_{f,s}).
\]
for units $u_1,\dots, u_{d}\in \Ma$ and any integer $k\in \Z$.
Finally, we set $QL_s(a_1,\dots, a_{d})=0$ whenever $a_i\neq a_j$ for $i\neq j$, $a_1,\dots, a_d\in \Ma^*$.
\end{defin}

Let $\xi=u\pi$ with $u\equiv 1 \pmod{\pi^s}$ and $g\in \Lambda_{\xi}$. Let $\epsilon$ be the unit associated to $u$ from Section \ref{Varying the uniformizer of $K$}.  We may show, by modifying the argument in \cite{koly} Page 343 to our situation, that when $e_{g,s}=[u^{-s}\epsilon]_{f,g}(e_{f,s})$ we have
\begin{equation}\label{lift-evaluated-in-eg}
QL_s(u_1,\dots, u_{d-1}, e_{g,s})=\frac{\det \left[\frac{\partial u_i}{\partial T_j}\right]}{u_1\cdots u_{d-1}}\,\frac{ \Tdc}{\ \xi^s\,l'_g(e_{g,s})\,e_{g,s}\ }\,\tau_{\epsilon}\ ,
\end{equation}
 where $\tau_{\epsilon}$ is a unit  in $\Zp$ such that $\epsilon\equiv \tau_{\epsilon}\pmod{\pi^s}$ (cf. \cite{koly} Lemma 7.1). This $\tau_{\epsilon}$ is uniquely determined by multiplication by a unit which is congruent to 1 $\pmod{ \pi^s}$.

In order to verify \eqref{lift-evaluated-in-eg} let $I/\Qp$ be an unramified extension such that $\epsilon$ belongs to $\mathcal{O}_I$, the ring of integers of $I$. Let 
\begin{equation}\label{urm}
M'=I K_{\pi,s}\quad  \text{and}\quad \Ma'=M'\TT. 
\end{equation}
Then   the derivation $Q_s:=Q_{\Ma,s}:\mathcal{O}_s\to P_s/((\pi^s/\pi_1)P_s)$ from Section \ref{cann} extends to a derivation
\[
Q_s:\mathcal{O}_{\Ma'}\to \frac{P_s\mathcal{O}_{\Ma'}}{\frac{\pi^s}{\pi_1}P_s\mathcal{O}_{\Ma'}}\, ,
\]
by including $\mathcal{O}_{I}$ in the ring of constants of $Q_s$.

From \eqref{qmsl} and \eqref{general-derivation-def} we have
\[
Q_s(u_1,\dots, u_{d-1},e_{f,s})=\det \left[\frac{\partial u_i}{\partial T_j}\right]\,Q_s(T_1,\dots, T_{d-1},e_{f,s})=\det \left[\frac{\partial u_i}{\partial T_j}\right]\,\frac{\Tdc}{\pi^s\,l_f'(e_{f,s})}\, .
\]
Let $[u^{-n}\epsilon]_{f,g}(X)=t(X)$. Then
\[
QL_s(u_1,\dots, u_{d-1},e_{g,s})=QL_s(u_1,\dots, u_{d-1},e_{f,s})+QL_s\left(u_1,\dots, u_{d-1},\left( \frac{t(X)}{X}\right)(e_{f,s})\right).
\]
Thus, since $(t(X)/X)(e_{f,s})$ is a unit, by the definition of $QL_s$ and the fact that  $t(e_{f,s})=e_{g,s}$ we have
\begin{align*}
 & QL_s\left(u_1,\dots, u_{d-1},  \left(   \frac{t(X)}{X}\right)(e_{f,s})  \right)  \\
&=\left(\left(\frac{t'(e_{f,s})}{e_{f,s}}-\frac{t(e_{f,s})}{(e_{f,s})^2}\right)\bigg/ \frac{t(e_{f,s})}{e_{f,s}}\right)Q_s(u_1,\dots,u_{d-1},e_{f,s})\\
&=\left(
      \left(\frac{t'(e_{f,s})}{e_{f,s}}-\frac{t(e_{f,s})}{(e_{f,s})^2} \right)
        \bigg/ \frac{t(e_{f,s})}{e_{f,s}}\right)
\frac{\det \left[\frac{\partial u_i}{\partial T_j}\right]}{u_1\cdots u_{d-1}}\,\frac{\Tdc}{\pi^s\,l_f'(e_{f,s})}\\
&=\frac{t'(e_{f,s})\,\Tdc}{e_{g,s}\,l_f'(e_{f,s})\pi^s}\frac{\det \left[\frac{\partial u_i}{\partial T_j}\right]}{u_1\cdots u_{d-1}}-QL_s(u_1,\dots, u_{d-1},e_{f,s}).
\end{align*}
Observing that
\[
l_g(t(X))=u^{-s}\epsilon\,l_f(X) 
\quad \Longrightarrow \quad  
l_g'(t(X))\,t'(X)=u^{-s}\epsilon \,l_f'(X),
\]
then, since $\epsilon-\tau_{\epsilon}\equiv 0\pmod{\pi^s}$, we have
\[
\frac{t'(e_{f,s})}{e_{g,s}\, l_f'(e_{f,s})\pi^s}
=\frac{\epsilon}{u^s\,\pi^sl_g'(e_{g,s})\,e_{g,s}}
=\frac{\epsilon}{\xi^s\,l_g'(e_{g,s})\,e_{g,s}}
=\frac{\tau_\epsilon}{\xi^s\,l_g'(e_{g,s})\,e_{g,s}},
\]
from which \eqref{lift-evaluated-in-eg} follows. Furthermore, we obtain the following
\begin{lemma}\label{lema-impt}
Let $\La=\Ka_{\pi,n}$ and $QL_n$ be the lift from Definition \ref{lift}. Let $\xi=\pi u$ with $u\equiv 1\pmod{\pi^n}$ and take $g\in \Pi\Lambda_{\xi}$. Let $\epsilon$ be the unit associated to $u$ from Section \ref{Varying the uniformizer of $K$}. Set $e_{g,n}=[u^{-n}\epsilon]_{f,g}(e_{f,n})$. The following identity holds
\[
\big(\,\{u_1,\dots, u_{d-1},e_{g,n}\}\, ,\,x\,\big)_{\La,n}= \left[\Tr\left(\,QL_n(u_1,\dots, u_{d-1},e_{g,n})\, l_f(x)\,\right)\right]_f\,(e_{f,n})
\]
for all units $u_1\dots, u_{d-1}$ in $\La$ and  all $x\in F_f(\mu_{\La})$  such that $v_{\La}(x)\geq 2\,v_{\La}(p)/(\varrho\,(q-1))$.
\end{lemma}
\begin{proof}
Let $x\in F(\mu_{\La})$
such that $v_{\La}(x)\geq 2\,v_{\La}(p)/(\varrho\,(q-1))$. Set
\[
y=[\epsilon]_{f,g}(x)\ominus_g l_g^{-1}\big(\,(\epsilon-\tau_{\epsilon})\,l_f(x)\,\big),
\]
where $\ominus_g$ is subtraction in the formal group $F_g$ and $\tau_{\epsilon}$ is the unit  in $\Zp$ such that $\epsilon\equiv \tau_{\epsilon}\pmod{\pi^n}$ that was described above (for $s=n$). Let  $f^{(n)}(z)=x$. 
It follows, basically using Kolyvagin's same argument in \cite{koly} Proposition 7.26, that this $y$
belongs to $F_g(\mu_{\La})$ and, furthermore,  $l_g(y)=\tau_{\epsilon}l_f(x)$. We will show that 
\begin{equation}\label{dictio}
(\alpha,y)_{g,\La,n}=[u^{-n}\epsilon]_{f,g}( \,(\alpha,x)_{\La,n} \,)\quad (\,\forall \alpha\in K_d(\La)).
\end{equation}
 Indeed, observe first that
\[
[\xi^n]_g \left( \, [u^{-n}]_g \, \left( \, 
[\epsilon]_{f,g}(z)\, \ominus_g\, l_g^{-1}\left(\frac{(\epsilon-\tau_{\epsilon})}{\pi^n}\,l_f(x)\right)
  \, \right)  \, \right)
  =y,
\]
where $((\epsilon-\tau_{\epsilon})/\pi^n)l(x)\in \pi_1^2 \mathcal{O}_{\Ma'}$  for $\Ma'$ as in \eqref{urm} with $s=n$; therefore $l_g^{-1}$ converges
on $((\epsilon-\tau_{\epsilon})/\pi^n)l(x)$ by Lemma \ref{local-iso-log}. Let $\alpha\in K_d(\La)$, $\zeta:=(\alpha,x)_{\La,n}$ and
\[
w:=[u^{-n}]_g \, \left( \, 
[\epsilon]_{f,g}(z)\, \ominus_g\, l_g^{-1}\left(\frac{(\epsilon-\tau_{\epsilon})}{\pi^n}\,l_f(x)\right)
  \, \right).
\]
Then $g^{(n)}(w)=y$ and so, from the very definition of $(\alpha,y)_{g,\La,n}$, we have
\[
(\alpha,y)_{g,\La,n}=\Upsilon_{\La}(\alpha)(w)\ominus_g w.
\]
Let $z':=\Upsilon_{\La}(\alpha)(z)$. Since  by \eqref{Frob-higher} $$\Upsilon_{\La}(\alpha)\big(\,[\epsilon]_{f,g}(z)\,\big)
=[\epsilon u^{\mathfrak{v}_{\La}(\alpha)}]_{f,g}\big(z'\big)\, ,$$ 
 then
\[
\Upsilon_{\La}(\alpha)(w)=[u^{-n}]_g
     \left( \,  
       [\epsilon\,u^{\mathfrak{v}_{\La}(\alpha)}]_{f,g}(z')
        \, \ominus_g \, 
           l_g^{-1} \left( \frac{(\epsilon\,u^{\mathfrak{v}_{\La}(\alpha)}-\tau_{\epsilon})}{\pi^n}\,l_f(x)\right)
            \right)\, ,
\]
and so
\begin{align*}
(\alpha,y)_{g,\La,n}
=[u^{-n}]_g \bigg(
     &
        \big\{ \, 
            [\epsilon\,u^{\mathfrak{v}_{\La}(\alpha)}]_{f,g}(z')
         \, \ominus_g \, [\epsilon]_{f,g}(z')
           \, \big\}    \\
         & \,  \oplus_g \,
         \big\{\,
           [\epsilon]_{f,g}(z')\ominus_g [\epsilon]_{f,g}(z)
                \, \big\} 
        \ominus_g
           l_g^{-1} 
              \left( 
              \frac{\epsilon(\,u^{\mathfrak{v}_{\La}(\alpha)}-1)}{\pi^n}\,l_f(x)
               \right)
               \bigg)\, .
\end{align*}
Then
\[
(\alpha,y)_{g,\La,n}=[u^{-n}\epsilon]_{f,g}(\zeta)
      \oplus_g
      [u^{-n}]_g \left(\,
      [\epsilon(u^{\mathfrak{v}_{\La}(\alpha)}-1)]_{f,g}(z')
      \ominus_g
      l_g^{-1}\left( \frac{\epsilon(u^{\mathfrak{v}_{\La}(\alpha)}-1)}{\pi^n}\, l_f(x)\right) 
      \,\right).
\]
But noticing that
\[
[\epsilon(u^{\mathfrak{v}_{\La}(\alpha)}-1)]_{f,g}(z')
      =
      l_g^{-1}\left( \frac{\epsilon(u^{\mathfrak{v}_{\La}(\alpha)}-1)}{\pi^n}\, l_f(x)\right),
\]
which follows by applying $l_g$ on both sides,  we obtain \eqref{dictio}.

Now let $e_{g,n}=[u^{-n}\epsilon]_{f,g}(e_{f,n})$. It follows from Proposition \ref{Lemma-Ev-torison} that
\begin{align*}
&\big(\,\{u_1,\dots, u_{d-1},e_{g,n}\}\,,\,y\,\big)_{g,\La,n}\\
&=\left[\Tr\left(\,\frac{\det \left[\frac{\partial u_i}{\partial T_j}\right]}{u_1\cdots u_{d-1}}\,\frac{\Tdc}{\,\xi^n\,l'_g(e_{g,n})\,e_{g,n}\,}\, l_g(y)\,\right)\right]_g(e_{g,n})\\
&=\left[\Tr\left(\,\frac{\det \left[\frac{\partial u_i}{\partial T_j}\right]}{u_1\cdots u_{d-1}}\,\frac{\,\Tdc}{\,\xi^n\,l'_g(e_{g,n})\,e_{g,n}\,}\,\tau_{\epsilon}\, l_f(x)\,\right)\right]_g\,(e_{g,n}).
\end{align*}
According to \eqref{lift-evaluated-in-eg} this last expression coincides with
\[
\Tr\left(\, QL_n\big(\,u_1,\dots, u_{d-1},e_{g,n}\,\big)\, l_f(x)\,\right),
\]
where here we understand $QL_n$ as the lifted logarithmic derivative from Definition \ref{lift}. Thus in light of \eqref{dictio} we obtain
\[
\big(\,\{u_1,\dots, u_{d-1},e_{g,n}\}\,,\,x\,\big)_{\La,n}
=\left[
\Tr\left(\, QL_n\big(\,u_1,\dots, u_{d-1},e_{g,n}\,\big)\, l_f(x)\,\right)\right]_f(e_{f,n}).
\]

\end{proof}

In order to refine the results of Theorem \ref{Partial-Iwasawa-formulas}, which is the content of Corollary  \ref{Iwa-Id-Wiles group} below, we need to do a further analysis 
of the behavior of the formulas and logarithmic derivatives at elements of $K_d(\La)$ of the form $\{u_1,\dots, u_d\}$, where $u_1,\dots, u_d$ are units in $\La$.  The result of this analysis is  stated in
Lemma \ref{units-Iwasawa-formulas} which, combined with 
Lemma \ref{lema-impt}, will finally yield Corollary  \ref{Iwa-Id-Wiles group}.

As a preamble for Lemma \ref{units-Iwasawa-formulas}, we start with the following considerations. Let $\Ma=\Ka_{\pi,s}$. From Definition \ref{lift} and Equation \eqref{units-map-q} both $QL_s$ and its lift
are defined $\text{mod}\,(\pi^s/\pi_1)P_s$ when evaluated at elements  of the form
$\{u_1,\dots, u_d\} \in K_d(\Ma)$ with $u_1,\dots,u_d\in \OMa^*=\{x\in\Ma:v_{\Ma}(x)=0\}$. More specifically,
\begin{equation}\label{lifty-unity}
QL_s(u_1,\dots, u_d)\in P_s/\big(\,(\pi^s/\pi_1)P_s\,\big).
\end{equation}
In particular, we observe from the very definitions that $QL_s$ and its lift coincide when evaluated at these specific elements.

Therefore, if we let $UK_d(\Ma)$ be the subgroup of $K_d(\Ma)$ generated by all the elements
of the form $\{u_1,\dots, u_d\} \in K_d(\Ma)$ for units $u_1,\dots,u_d\in \OMa^*$
then by restricting $QL_s$ to $UK_d(\Ma)$ we obtain a map
\begin{equation}\label{restricted-lift}
QL_s:UK_d(\Ma)\longrightarrow \frac{P_s}{\frac{\pi^s}{\pi_1}P_s}.
\end{equation}

This map satisfies an analogous compatibility   relation to Proposition \ref{levelQL}.
\begin{prop}\label{compat-units}
Let $\La=\Ka_{\pi,n}$. For any $t\geq s \geq n$ the following diagram commutes
\[
\begin{CD}
 UK_d(\La_t) @>{ QL_{t} }>> \frac{P_t}{\frac{\pi^t}{\pi_1}P_t}   \\
@V{N_{t/s}}VV @VV{\normalfont \text{Tr}_{t/s}}V \\
UK_d(\La_s) @>>{QL_s}> \frac{P_s}{\frac{\pi^s}{\pi_1}P_s}.
\end{CD}
\]
The fact that $N_{t/s}( UK_d(\La_t))\subset UK_d(\La_s) $ is a consequence of \cite{Morrow} page 568.
\end{prop}

\begin{proof}
Just  as in the proof of Proposition \ref{levelQL} we assume $t=s+1$. 
By Equation \eqref{crux-prop} we have, in particular,  that for all $\alpha\in UK_d(\La_t)$
 \begin{multline}\label{anot}
          QL_t\left(N_{t/s}(\alpha)\right)=\left( \sum \, \tau_t(g)g\, 
                                               \right)QL_t(\alpha)=\\
                                          =\left( \sum g\right)QL_t(\alpha)
                                            + \left( \sum \big(\tau_t(g)-1\big)g \right) QL_t(\alpha),
     \end{multline}
       where each $\sum$ is taken over all $g\in \text{Gal}(\La_t/\La_s)$. Since $QL_t(\alpha)\in P_t/((\pi^t/\pi_1)P_t)$
       for $\alpha\in UK_d(\La_t)$ (cf. Equations \eqref{lifty-unity} and \eqref{restricted-lift}), and $ \sum \big(\tau_t(g)-1\big)g $ 
       takes $P_t$ to  $\frac{\pi^s}{\pi_1}P_s\OZ_{\La_t}$ by Proposition \ref{imagedelta}, then
     \begin{align*}
     QL_s\big(\,N_{t/s}(\alpha) \,\big)&=QL_t\big(\,N_{t/s}(\alpha)\,\big) \pmod{ \frac{\pi^s}{\pi_1}P_s}\\
                       &=\text{Tr}_{t/s}\big(\,QL_t( \alpha)\,\big),     
     \end{align*}
          where the first equality follows from Proposition \ref{extension-Q} and the second from Equation \eqref{anot}.
    
\end{proof}

Let $\La=\Ka_{\pi,n}. $ In order to obtain an analogous result to Theorem \ref{Partial-Iwasawa-formulas}
for the group $UK_d(\La)$, we need to define
\[
UK_d(\La)':=\bigcap_{m\geq n}UK_{d}(\La_m).
\]
With this notation we now have the following result.

\begin{lemma}\label{units-Iwasawa-formulas}
Let $\La=\Ka_{\pi,n}$ and  $QL_n$ be the map from Equation \eqref{restricted-lift}. The following identity holds:
\begin{equation}\label{formula-in-units}
(\alpha, x)_{\La,n}= \left[\,\Tr  \big(\,QL_n(\alpha)\,l_f(x)\,\big)\,\right]_f(e_{f,n})
\end{equation}
for all $\alpha \in UK_d(\La)'$ and 
all $x\in F(\mu_{\La})$ such that $v_{\La}(x)\geq 2\,v_{\La}(p)/(\varrho\,(q-1))$.
\end{lemma}
\begin{proof}
First of all, notice that for $\La=\Ka_{\pi,n}$ we have $n=r=r'$; for $r$ and $r'$ as in Section \ref{The canonical logarithmic derivatives}. Borrowing the notation from Theorem  \ref{Lubinazo} we have,  from
Equation \eqref{key-id-res} restricted to $ UK_d(\La_t)$, the following:  for  all $\epsilon\in UK_d(\La_t)$ and  all $x\in F(\mu_{\La})$
\begin{equation}\label{nece-ta}
(N_{\La_t/\La}(\epsilon),x)_{\La,n}^1=\mathbb{T}_{\La_t/K}(QL_t(\epsilon)l_f(x)) \pmod{\pi^nC}. 
\end{equation}
Here is important to observe that since we are restricting ourselves to elements $\alpha\in UK_d(\La_t)$,
then the map $QL_t(\alpha)$ is defined mod$(\pi^n/\pi_1)P_{\Ma}$ (cf. Equation  \eqref{units-map-q}).
Furthermore,  according to Equations \eqref{lifty-unity} and  \eqref{restricted-lift} the map $QL_t$ coincides with its lift at $UK_d(\La_t)$. Therefore in Equation \eqref{nece-ta} the map $QL_t$ can be taken to be the lift of $QL_t$.

 On the other hand, from  Proposition \ref{compat-units} (applied to $s=r'=n$) we have that
\begin{equation}\label{trazalll-ta}
QL_n\big(\,N_{t/n}(\epsilon)\,\big)= \normalfont \text{Tr}_{\La_t/\La}\big(\,QL_t(\epsilon)\,\big) \pmod{\frac{\pi^n}{\pi_1}P_n};
\end{equation}
taking into account that $\text{Tr}_{t/n}=\text{Tr}_{\La_t/\La} $ since $\La_n=\La$. 

Suppose for a moment that
\begin{equation}\label{dualita}
v_{\La}(x)\geq  2v_{\La}(p)/(\varrho(q-1))
\quad \Longrightarrow \quad 
\Tr\left(\,\left(\frac{\pi^n}{\pi_1}P_n\right)\, l_f(x)\,\right)\subset \pi^n C.
\end{equation}
Then  multiplying both sides of \eqref{trazalll-ta} by $l_f(x)$, for $x\in F(\mu_{\La})$ such that $v_{\La}(x)\geq  2v_{\La}(p)/(\varrho(q-1))$, and  
then  applying $\Tr$ we obtain
\[
\Tr\big(\,QL_n\big(\,N_{t/n}(\epsilon)\,\big)\cdot l_f(x)\,\big)= \mathbb{T}_{\La_t/K}\big(\,QL_t(\epsilon)\cdot l_f(x)\,\big) \pmod{\pi^nC}.
\]
Therefore, combining this with \eqref{nece-ta} we have, for all $\epsilon\in UK_d(\La_t)$ and all $x\in F(\mu_{\La})$ such that $v_{\La}(x)\geq  2v_{\La}(p)/(\varrho(q-1))$, that
\begin{equation}\label{clave-id-ta}
(N_{\La_t/\La}(\epsilon),x)_{\La,n}^1=\Tr\big(\,QL_n\big(\,N_{t/n}(\epsilon)\,\big)\cdot l_f(x)\,\big) \pmod{\pi^nC}.
\end{equation}
 To finally prove the identity, take $\alpha \in UK_d(\La)'=\cap_{m\geq n}UK_d(\La_m)$. Therefore there exists an $\epsilon\in UK_d(\La_t)$ such that $\alpha=N_{\La_t/\La}(\epsilon)$; recalling that $t$ is a large enough integer as it is specified in Theorem \ref{Lubinazo}. Therefore applying the identity \eqref{clave-id-ta} to $\alpha=N_{\La_t/\La}(\epsilon)$ we obtain
 \[
 (\alpha, x)_{\La,n}^1= \Tr  \big(\,QL_n(\alpha)\,l_f(x)\,\big) \pmod{\pi^nC},
 \]
 for all $\alpha\in K_d(\La)'$ and all $x\in F(\mu_{\La})$ such that $v_{\La}(x)\geq  2v_{\La}(p)/(\varrho(q-1))$. Recalling the notation from Equation \eqref{notation-kummer-pairing}, the above identity is equal to 
 \eqref{formula-in-units}.
 
In conclusion, it  remains only to show \eqref{dualita}. This follows by noticing first that $\pi_1^2\OLa$ if and only if  $v_{\La}(x)\geq  2v_{\La}(p)/(\varrho(q-1))$. Second, by noticing that
 $\pi_1^2\OLa$
is the dual of $$\frac{\pi^n}{\pi_1}P_{\La}=\frac{\pi^n}{\pi_1^2}\cdot\frac{1}{D(\La/K)}$$ with respect to the pairing $(y,w)\to \Tr(yw)$ mod $\pi^nC$. Finally, by observing that 
$x\in \pi_1^2\OLa$ if and only if $l_f(x)\in \pi_1^2\OLa$,  by
 Lemma \ref{local-iso-log}.
 
\end{proof}

From   Lemma \ref{lema-impt} and Lemma \ref{units-Iwasawa-formulas} we obtain the following

\begin{coro}\label{Iwa-Id-Wiles group}
Let $\La=\Ka_{\pi,n}$ and  $QL_n$ be the lift from Definition \ref{lift}. Then
\begin{equation}\label{final=}
(\alpha, x)_{\La,n}= \left[\,\Tr  \big(\,QL_n(\alpha)\,l_f(x)\,\big)\,\right]_f(e_{f,n})
\end{equation} 
for all $\alpha\in K_d(\La)'$ and all $x\in F(\mu_{\La})$ such that $v_{\La}(x)\geq 2\,v_{\La}(p)/(\varrho\,(q-1))$.
\end{coro}
\begin{proof}
By Lemma \ref{units-Iwasawa-formulas}, the identity \eqref{final=} holds for all
$\alpha \in UK_d(\La)'$. On the other hand, by Lemma \ref{lema-impt}, \eqref{final=} also holds for all $\alpha\in K_d(\La)$ of the form 
\begin{equation}\label{special-elemi}
\alpha=\{u_1,\dots, u_{d-1},e_{g,n}\}
\end{equation}
with $u_1,\dots, u_{d-1}\in\OLa^*=\{x\in \La\,:\,v_{\La}(x)=0\}$, and $e_{g,n}$ as in Lemma \ref{lema-impt}.
Therefore, it is enough to show that $K_d(\La)'$ is generated by $UK_d(\La)'$
 and all the elements  of the form \eqref{special-elemi}. This is the content of the lemma below.
 \end{proof}
 
 \begin{lemma}
 Keeping the notation from Corollary \ref{Iwa-Id-Wiles group} and Lemma \ref{lema-impt}  we have
 \begin{equation*}\label{norm-iden-key}
K_d(\La)'=UK_d(\La)'\cdot \langle \{u_1,\dots,u_{d-1}, e_{g,n}\}\,|\,u_1,\dots, u_{d-1}\in \OLa^*\rangle, 
  \end{equation*}
  where $\OLa^*=\{x\in \La\,:\,v_{\La}(x)=0\}$.
 \end{lemma}
 \begin{proof}
Let $t\geq n$. We will show that the result is a consequence of the commutative diagram
\[
\begin{CD}
 K_d(\La_t) @>{ \partial_{\La_t} }>> K_{d-1}(\overline{\La_t})   \\
@V{N_{t/n}}VV @VV{N_{\overline{\La_t}/\overline{\La}}}V \\
K_d(\La) @>>{\partial_{\La}}> K_{d-1}(\overline{\La}),
\end{CD}
\]
and the exact sequences
 \begin{equation}\label{exact-seq-L}
0\to UK_d(\La)\to K_d(\La)\xrightarrow[]{\partial_{\La}} K_{d-1}(\overline{}\La)\to 0
 \end{equation}
and
 \begin{equation}\label{exact-seq-L_t}
0\to UK_d(\La_t)\to K_d(\La_t)\xrightarrow[]{\partial_{\La_t}} K_{d-1}(\overline{}\La_t)\to 0
 \end{equation}
(cf. \cite{Morrow} Page 567 and 568), where $\partial_{\La_t}$ and $\partial_{\La}$ denote the boundary maps
 ( cf. \cite{Morrow} Page 567 or  \cite{Zhukov} Chapter 6, Section 6.4.1), and  $\overline{\La_t}$ and $\overline{\La}$ the residue fields of $\La_t$ and $\La$, respectively.
 
 Indeed,  let $\phi \in K_d(\La)'$.
 Then for every $t\geq n$ there exists an $\alpha\in K_d(\La_t)$ such that $N_{t/n}(\alpha)=\phi$. From the exact sequence \eqref{exact-seq-L}
we have for this $\phi$ that there exist units $u_1,\dots, u_{d-1}\in \OLa^*$ and an integer $m\geq0$ such that 
\[
\partial_{\La} (\phi)=\partial (u_1,\dots,u_{d-1}, (e_{g,n})^m);
\]
recalling that we are taking $e_{g,n}$ as a uniformizer for $\La=\Ka_{\pi,n}$. On the other hand, $\beta:=\{u_1,\dots,u_{d-1}, (e_{g,n})^m\}$ is a norm from $\La_t$: $$N_{t/n}(u_1,\dots, (e_{g,t})^m)=\{u_1,\dots,N_{t/n}(e_{g,t})^m \}=\beta$$ ( since $N_{t/n}(e_{g,t})=e_{g,n}$ by \eqref{real-norm-torsion} ). Therefore the element
\[
\delta:=\alpha\cdot \{u_1,\dots, (e_{g,t})^m\}^{-1}\in K_d(\La_t)
\]
satisfies that $\partial_{\La}(N_{t/n}(\delta))=\partial_{\La}(\phi\cdot \beta^{-1})=0$. But from the commutative diagram we then have
$0=\partial_{\La}(N_{t/n}(\delta))=\partial_{\La_t}(\delta)$ (since $\overline{\La_t}=\overline{\La}$), which implies $\delta \in UK_d(\La_t)$ by virtue of  the exact sequence \eqref{exact-seq-L_t}. Therefore $\phi\cdot \beta^{-1}=N_{\La_t/\La}(\delta)\in N_{t/n}(UK_d(\La_t))$. Thus, the element $\sigma:=\phi\cdot \beta^{-1}$ belongs to 
$N_{t/n}(UK_d(\La_t))$ for all $t\geq n$, i.e., $\sigma\in UK_d(\La)'$. In conclusion, we have shown that every element $\phi\in K_d(\La)'$ can be put in the form
 $\sigma \cdot \beta$, as claimed in the statement of the lemma.
\end{proof}

\subsection{Generalized Iwasawa formulas}
We arrived finally at the main result in Section \ref{Comparison of formulas}, namely,  Theorem \ref{Iwasawa-General}. This theorem is as a generalization of Iwasawa's formulas \eqref{Iwasawa} to the Kummer pairing associated to a Lubin-Tate formal group and the higher local field $\La=\Ka_{\pi,n}$. These reciprocity laws coincide with the formulas of Zinoviev (cf. \cite{Zinoviev} Theorem 2.2 ) and  Kurihara (cf. \cite{Kurihara} Theorem 4.4) for the generalized Hilbert symbol and the field $\La=\Qp(\zeta_{p^n})\TT$.

Recall that $\Ka$ denotes the field $K\TT$. In the deduction of  Theorem \ref{Iwasawa-General} we shall use the following result.
\begin{lemma}\label{Wiles-group-gen}
Let $\La=\Ka_{\pi,n}$. The following isomorphism of groups holds
\[
\frac{\,N_{\La/\Ka}(K_d(\La))\,}{\, N_{\La/\Ka}(K_d(\La)')\,}
\, \simeq \,
1+\pi^n\,C.
\]
More specifically, the above quotient is topologically generated by the elements
\[
\{T_1,\dots, T_{d-1},u\} \quad \text{for all $u\in 1+\pi^n C$}.
\]
\end{lemma}
\begin{proof}
By Higher Class Field Theory we have, for $m\geq 1$, that
\begin{equation}\label{HCFT-iso}
K_d(\Ka)/N_{\Ka_{\pi,m}/\Ka}(K_d(\Ka_{\pi,m}))
\, \simeq\,
 \text{Gal}(\Ka_{\pi,m}/\Ka ).
\end{equation}
But by \eqref{Lubin-Tate-isom}
\[
\text{Gal}(\Ka_{\pi,m}/\Ka)\,
\,\simeq  \, 
(C/\pi^mC)^{\times}
\, \simeq  \,
k_{K}^{\times}\times \big(\,(1+\pi C)/(1+\pi^m C)\,\big).
\]
Therefore, for all $m\geq n$,
\[
N_{\Ka_{\pi,n}/\Ka}(K_d(\Ka_{\pi,n}))/N_{\Ka_{\pi,m}/\Ka}(K_d(\Ka_{\pi,m}))\simeq (1+\pi^n C)/(1+\pi^m C).
\]
Since $N_{\Ka_{\pi,n}/\Ka}(K_d(\Ka_{\pi,n}')=\cap_{m\geq n} N_{\Ka_{\pi,m}/\Ka}(K_d(\Ka_{\pi,m}))$, this implies
\[
N_{\Ka_{\pi,n}/\Ka}(K_d(\Ka_{\pi,n}))/N_{\Ka_{\pi,n}/\Ka}(K_d(\Ka_{\pi,n})')
\, \simeq \,
 1+\pi^n C.
\]

As for the second statement in this lemma, notice first that we can replace the Milnor K-groups $K_d(\cdot)$ by
 the topological Milnor K-groups  $K_d^{\text{top}}(\cdot)$   endowed with the Parshin topology (cf. \cite{Zhukov} Section 6). For simplicity in the notation let us consider the case $d=2$, namely $\Ka=K\{\{T\}\}$ and $\Ka_{\pi,m}=K_{\pi,m}\{\{T\}\}$ ($m\geq 1$).    For the group $K_d^{\text{top}}(\Ka)$ we have the following set of topological generators (cf. \cite{Zhukov} Section 6.5) associated to the local uniformizers $T$ and $\pi$: 
\begin{enumerate}
\item \label{gen-1} $\{T,\pi\}$, 
\item\label{gen-2} $\{T,\theta \}$,  
\item\label{gen-2} $\{ \pi,\theta \}$, 
\item\label{gen-2} $\{ \pi,1+\theta\,\pi^k\,T^j \}$,
\item\label{gen-2} $\{ T,1+\theta\,\pi^k\,T^j \}$, 
\end{enumerate}
where   $\theta\in \mu_{q-1}:=$ the group of $(q-1)$th roots of 1, $k>0$ and $j\in \Z$.
 By analyzing the isomorphism \eqref{HCFT-iso} we see that $N_{\Ka_{\pi,m}/\Ka}(K_d^{\text{top}}(\Ka_{\pi,m}))$ must be generated by
 elements of the form
 \begin{enumerate}
\item \label{gen-1} $\{T,\pi\}$ ,
\item\label{gen-2} $\{ \pi,1+\theta\,\pi^k\,T^j \}$, $j\in \Z$,
\item\label{gen-2} $\{ T,1+\theta\,\pi^k\,T^j \}$, $ 0\neq j \in \Z$,
\item\label{gen-2} $\{ T,1+\theta\,\pi^k\}$, $k\geq m$.
\end{enumerate}
Indeed, if we let $\mathfrak{M}_m$ be the subgroup of $K_d^{\text{top}}(\Ka)$ generated by the elements above, 
it follows that $N_{\Ka_{\pi,m}/\Ka}(K_d^{\text{top}}(\Ka_{\pi,m}))\subset \mathfrak{M}_m$ and $K_d^{\text{top}}(K)/\mathfrak{M}_m\simeq (C/\pi^mC)^{\times}$, which implies $N_{\Ka_{\pi,m}/\Ka}(K_d^{\text{top}}(\Ka_{\pi,m}))=\mathfrak{M}_m$.

Therefore $K_d^{\text{top}}(\Ka)/N_{\Ka_{\pi,m}/\Ka}(K_d^{\text{top}}(\Ka_{\pi,m}))$ is generated by elements of the form $\{T,u\}$ with $u\in (C/\pi^mC)^{\times}$. This implies that the quotient group $$N_{\Ka_{\pi,n}/\Ka}(K_d^{\text{top}}(\Ka_{\pi,n}))/N_{\Ka_{\pi,m}/\Ka}(K_d^{\text{top}}(\Ka_{\pi,m}))$$ is generated by elements of the form $\{T,u\}$ with $u\in (1+\pi^n\,C)/(1+\pi^m\,C)$ for all $m\geq n$. This concludes the proof.
\end{proof}

We can now prove the main result.

\begin{thm}\label{Iwasawa-General}
Let $\La=\Ka_{\pi,n}$ and $QL_n$ be the lift from Definition \ref{lift}. Then
\begin{equation}\label{Iwasawa-formula}
(\alpha,x)_{\La,n}=\left[\,\Tr \big( QL_n(\alpha)\, l_f(x) \big)\,\right]_f(e_{f,n})
\end{equation}
for all $\alpha\in K_n(\La)$ and $x\in F(\mu_\La)$ such that $v_{\La}(x)\geq 2\,v_{\La}(p)/(\varrho\,(q-1))$. 
\end{thm}

\begin{proof}
First of all, the identity \eqref{Iwasawa-formula} is true for all $\alpha \in K_d(\La)'$ by Corollary \ref{Iwa-Id-Wiles group}.

Now let $u\in 1+\pi^n C$ and set $\xi=\pi\,u$. Take $g\in \Pi \Lambda_{\xi}$ and let $e_{g,n}=[u^{-n}\,\epsilon ]_{g,f}\,(e_{f,n})$. Then $N_{L/K}(e_{g,n})=(-1)^{q^n-q^{n-1}}\xi=\xi$ since the irreducible monic polynomial for $e_{g,n}$ over $K$ is $g^{(n)}/g^{(n-1)}$; which has $q^{n}-q^{n-1}$ roots and constant coefficient $\xi$. By Lemma \ref{lema-impt} the identity \eqref{Iwasawa-formula} holds for $\{T_1,\dots, T_{d-1},e_{g,n}\}$. Since $$N_{\La/\Ka}(\{T_1,\dots T_{d-1},e_{g,n}\})=\{T_1,\dots, T_{d-1},\xi\}=\{T_1,\dots, T_{d-1},\pi\,u\},$$ 
and $u\in 1+\pi^n C$ was chosen arbitrarily, then by Lemma \ref{Wiles-group-gen} this implies the result for all $\alpha\in K_{d}(\La)$.

\end{proof}

\noindent\emph{Acknowledgments.}
The author would like to thank V. Kolyvagin for suggesting the problem treated in this article,  for  reading the manuscript and providing valuable comments and improvements. The author would also like to thank the anonymous  referee for the careful review of the paper, and for the corrections and  valuable comments for improvements.


\end{document}